\documentclass[11pt]{article}

\usepackage{epsfig,epsf,fancybox}
\usepackage{amsmath}
\usepackage{mathrsfs}
\usepackage{amssymb}
\usepackage{graphicx}
\usepackage{color}
\usepackage{multirow}
\usepackage{paralist}
\usepackage{verbatim}
\usepackage{galois}
\usepackage{algorithm}
\usepackage[noend]{algorithmic}
\usepackage{boxedminipage}
\usepackage{booktabs}
\usepackage{accents}
\usepackage{stmaryrd}
\usepackage[table]{xcolor}
\usepackage{hhline}
\usepackage{tikz}
\usetikzlibrary{arrows.meta,positioning,calc}

\usepackage{subfig}

\usepackage{natbib}

\usepackage{url}
\usepackage[colorlinks,linkcolor=magenta,citecolor=blue, pagebackref=true,backref=true]{hyperref}
\renewcommand*{\backrefalt}[4]{%
    \ifcase #1 \footnotesize{(Not cited.)}%
    \or        \footnotesize{(Cited on page~#2.)}%
    \else      \footnotesize{(Cited on pages~#2.)}%
    \fi}

\newtheorem{theorem}{Theorem}[section]

\newtheorem{lemma}[theorem]{Lemma}
\newtheorem{proposition}[theorem]{Proposition}

\newtheorem{definition}[theorem]{Definition}

\newtheorem{assumption}[theorem]{Assumption}

\numberwithin{equation}{section}

\newenvironment{proof}[1][Proof]{\par\noindent{\em #1.}\ }{\hfill $\Box$ \vskip 0.4cm}

\newcommand{\EE}{\mathbb{E}}

\newcommand{\grad}{\nabla}

\newcommand{\conv}{\textnormal{conv}}

\newcommand{\HH}{\mathbf H}

\newcommand{\argmin}{\mathop{\rm argmin}}

\newcommand{\DCal}{\mathcal{D}}
\newcommand{\ECal}{\mathcal{E}}
\newcommand{\GCal}{\mathcal{G}}
\newcommand{\HCal}{\mathcal{H}}

\newcommand{\OCal}{\mathcal{O}}
\newcommand{\PCal}{\mathcal{P}}

\newcommand{\XCal}{\mathcal{X}}

\newcommand{\br}{\mathbb{R}}

\newcommand{\ba}{\begin{array}}
\newcommand{\ea}{\end{array}}
\newcommand{\ACal}{\mathcal{A}}
\newcommand{\BCal}{\mathcal{B}}

\newcommand{\FCal}{\mathcal{F}}

\newcommand{\TCal}{\mathcal{T}}
\newcommand{\bp}{\mathbb{P}}

\begin{document}

\begin{center}

{\bf{\LARGE{Nonsmooth Optimization via Orthogonalized Momentum}}}

\vspace*{.2in}
{\large{
\begin{tabular}{c}
Lexiao Lai$^\dagger$ \and Tianyi Lin$^\diamond$ \and Jiayu Zhang$^\diamond$ \\
\end{tabular}
}}

\vspace*{.2in}

\begin{tabular}{c}
Department of Industrial Engineering and Operations Research$^\diamond$ \\
Columbia University \\ 
Department of Mathematics, The University of Hong Kong$^\dagger$
\end{tabular}

\vspace*{.2in}

\today

\vspace*{.2in}

\begin{abstract}
Modern real application problems involve matrix-valued parameters, yet conventional optimizers treat them as vectors, thereby motivating matrix-aware methods that exploit input-output geometry, such as Muon which orthogonalizes the momentum matrices before parameter updates. Its empirical success raises a \textit{conceptual} question: can orthogonalized momentum remain effective beyond smooth optimization? This paper studies this question for locally Lipschitz functions using a generalized derivative framework compatible with backpropagation. Our first contribution is to identify a key limitation: for every fixed momentum factor $\beta \in [0,1)$, Muon can fail to approach the global optimal solution of a convex Lipschitz objective from almost every initialization, when step sizes adapt to the full gradient history. The failure can occur even along bounded iterates. Our example is inspired by the one of~\citet{Parshakova-2026-Muon} which only covers $\beta \in [0, \frac{1}{2})$. Then, we show that the obstruction lies in fixed momentum rather than orthogonalization. Indeed, when the momentum factor is \textit{adaptive} and approaches $1$ together with a vanishing step size, Muon recovers asymptotic convergence for nonconvex nonsmooth optimization under the boundedness and regularity conditions. Moreover, we propose a MAGD method, which combines orthogonalized momentum with gradient and weights them based on their relative progress. MAGD retains asymptotic convergence in nonconvex settings and achieves a convergence rate of $O(\min\{m,n\}\epsilon^{-2})$ in convex settings. A lower bound is established to show the optimal dimension dependence. Experiments on synthetic problems, image classification, and LLM pretraining show that our method is a \textit{simple} and \textit{practical} alternative to Muon. Together, our results characterize when orthogonalized momentum fails without smoothness and how it can be made reliable and we hope that the analysis may be useful more broadly.
\end{abstract}
\end{center}

\section{Introduction} \label{sec:intro}
A common geometric principle underlying first-order methods is that first-order information is converted into an update through a chosen norm or regularizer. Classical algorithms have adopted Euclidean or coordinate-wise geometry after vectorizing the decision variables. This framework includes stochastic gradient descent (SGD), momentum methods, and adaptive methods such as AdaGrad, RMSProp, and Adam~\citep{Robbins-1951-Stochastic, Polyak-1964-Some, Nesterov-1983-Method, Sutskever-2013-Importance, Duchi-2011-Adaptive, Tieleman-2012-Lecture, Kingma-2015-Adam}. For matrix-valued variables, however, vectorization obscures their input-output structure. The operator norm instead directly measures how strongly a matrix amplifies its inputs. This observation has motivated matrix-aware methods whose geometry respects the role of a matrix as a linear map~\citep{Large-2024-Scalable, Bernstein-2024-Old, Bernstein-2025-Modular, Pethick-2025-Training}. Neural network training is a prominent application because most trainable parameters are naturally organized as matrix-valued layers, but the underlying problem is more general: first-order optimization in matrix-norm geometry.

A simple matrix-aware first-order method forms a momentum matrix and orthogonalizes it before taking a main step. This update is known as Muon~\citep{Jordan-2024-Muon, Liu-2025-Muon}. For a matrix variable $W_t \in \br^{m\times n}$, the inexact Muon recursion is
\begin{equation*}
G_t \in \DCal_f(W_t),\quad M_{t+1}=\beta_tM_t+(1-\beta_t)G_t,\quad \widetilde M_{t+1}=\tfrac{M_{t+1}}{\|M_{t+1}\|_F+\sigma_t} \quad W_{t+1} \in W_t - \eta_t \operatorname{polar}_{\delta_t} (\widetilde M_{t+1}),
\end{equation*}
where $\DCal_f$ represents the first-order information returned by backpropagation and $\operatorname{polar}_{\delta}(\widetilde M)$ is the set of all inexact orthogonalizations of $\widetilde M\in\br^{m\times n}$ within an inexactness factor $\delta\in[0,1)$. In practice, we apply a small number of polynomial iterations to $\widetilde M$ to realize $\operatorname{polar}_\delta$, which preserves the singular subspaces of $M$ while pushing its nonzero singular values toward $1$~\citep{Jordan-2024-Muon, Amsel-2026-Polar}.

The empirical success of Muon raises a conceptual question: \textit{can orthogonalized momentum remain effective for nonsmooth optimization?} Recent analyses have shown the convergence of inexact Muon with finitely many inner iterations for smooth optimization~\citep{Shulgin-2026-Beyond,Kim-2026-Convergence}. Smoothness provides a local upper bound on the objective, ensuring that a sufficiently small step along a well-aligned direction produces a predictable decrease. This argument is unavailable for locally Lipschitz objectives, whose behavior may change abruptly at nonsmooth points \citep{wang2005subdifferentiability}. Orthogonalization creates an additional difficulty since it removes the magnitude information from the momentum matrix. Even an arbitrarily small nonzero momentum can produce a full operator-norm step. It is thus unclear whether momentum stabilizes these normalized updates or leads to nonconvergence. 

The recent work of \citet{Parshakova-2026-Muon} gives a negative answer for convex objectives. Indeed, for diagonal cases, Muon reduces to signed momentum~\citep{Bernstein-2018-Signsgd}, allowing the counterexamples for sign-based methods to be embedded into the matrix recursion. Based on this insight, they construct convex Lipschitz objectives on which Muon with fixed momentum does not converge. One result covers every fixed momentum factor $\beta \in [0, 1)$ but restricts the step size schedule and initialization, while another result allows the step size chosen from the full gradient information but only covers $\beta \in [0, \frac{1}{2})$. They also combine error feedback with Muon and restores the convergence for convex objectives. Their results establish a novel obstruction but leave its precise source unresolved. In particular, does almost everywhere failure extend to every fixed factor $\beta \in [0,1)$? Can an adaptive momentum schedule restore convergence, even when orthogonalization is computed inexactly? Can one obtain a finite-time guarantee with the optimal dimension dependence while retaining the effectiveness of orthogonalized momentum?

We study these questions in a nonconvex and locally Lipschitz setting. We formulate our analysis using conservative fields, set-valued mappings that encode generalized first-order information for nonsmooth functions.~\citep{Bolte-2021-Conservative}. This framework accommodates the outputs of backpropagation while retaining a chain rule. We identify the nonconvergence counterexamples and derive convergence results using conservative fields. The convex Lipschitz function class is a simple and clean benchmark: every stationary point is globally optimal, Lipschitz continuity bounds the subgradients, and Euclidean subgradient descent has the optimal $O(\epsilon^{-2})$ worst-case dependence~\citep{Nemirovski-1983-Problem,Beck-2017-First,Nesterov-2018-Lectures}. Nonconvergence counterexample is established on this class and cannot be attributed to a nonconvex landscape or unbounded first-order information.

The negative result comes from the interaction between the fixed momentum and orthogonalization. At a nonsmooth point, the first-order information can change abruptly. Then, a fixed momentum factor transmits this change at a rate independent of the step size, while orthogonalization turns every nonzero momentum into a normalized direction. Our construction exploits this behavior to preserve a nonzero component of the objective, regardless of chosen step sizes. For the positive results, we synchronize the momentum and parameter updates by setting
\begin{equation} \label{eq:schedule}
\beta_t = 1-\tau\eta_t,\quad \eta_t>0,\quad \eta_t \to 0,\quad \sum_{t=0}^{+\infty} \eta_t = +\infty.
\end{equation}
Related one-timescale couplings appear in stochastic heavy-ball schemes~\citep[Eq.~(4)]{Gadat-2018-Stochastic}, nonsmooth Adam-family methods~\citep{Xiao-2024-Adam}, and, for $\tau=1$, a variance-reduced Muon analysis~\citep[Theorem~4.2]{Chang-2025-Convergence}. Under the schedule \eqref{eq:schedule}, the momentum and parameter updates evolve on the same time scale and controlled approximation errors vanish asymptotically. This yields a Lyapunov function $\varphi(W,M)=f(W)+\tau^{-1}\|M\|_{\mathrm{nuc}}$, for which the conservative field chain rule and operator-nuclear norm duality yield $\frac{\mathrm{d}}{\mathrm{d}s}\varphi(W(s),M(s))=-\|M(s)\|_{\mathrm{nuc}}$. Coupling the two updates restores a dissipative structure which is absent under fixed momentum and yields the asymptotic convergence. To obtain a finite-time guarantee for convex objectives, we propose MAGD, which combines orthogonalized momentum with gradient descent and weights them based on their relative progress. The weighting rule is inspired by the one of~\citet{Cesa-2006-Prediction} and allows our method to inherit the safeguard without modifying the Muon direction.

\paragraph{Contributions.} The contribution of this paper consists in studying when orthogonalized momentum fails and how it can be made reliable beyond smoothness, and in establishing matching upper and lower complexity bounds for convex Lipschitz optimization in operator-norm geometry. In further detail:
\begin{enumerate}
\item We establish a sharp separation between fixed and adaptive momentum beyond smoothness. For every fixed $\beta \in [0,1)$, Muon fails to approach the global optimal solution of a convex and Lipschitz objective from almost every initialization, even when the step sizes adapt to full gradient history. The failure can occur even along bounded iterates. By using the coupled schedule in Eq.~\eqref{eq:schedule} and controlled inner approximation, inexact Muon converges asymptotically and the objective values converge to the global optimal value for convex objectives. 
\item We establish asymptotic and finite-time guarantees for MAGD. Under the conservative field model, an appropriate diminishing step size schedule yields stationarity of every cluster point as well as convergence of the objective values. For convex Lipschitz objectives on $\br^{m \times n}$, MAGD can return an $\epsilon$-optimal point within $O(\min\{m,n\}\epsilon^{-2})$ iterations. A lower bound is then established to show the optimal dimension dependence. Experiments on synthetic problems, image classification, and LLM pretraining show that our method is a simple and practical alternative to Muon. 
\end{enumerate}

\paragraph{Further related work.} In addition to the aforementioned works, we provide a few additional remarks regarding related work on Muon and related methods, nonsmooth optimization, lower bound analysis, and normalized and sign methods.

The existing theory of Muon and related scale-invariant methods has developed along two lines. Early analyses consider the exact polar update and establish convergence for several variants involving momentum, constraints, star-convexity, weight decay and variance reduction~\citep{Li-2025-Note, Shen-2025-Convergence, Sato-2025-Convergence, Sfyraki-2025-Lions,Chang-2025-Convergence}. Recent works have moved toward the inexact orthogonalization used in practice.~\citet{Shulgin-2026-Beyond} quantified how approximation error interacts with the step size and momentum, ~\citet{Kim-2026-Convergence} proved convergence with finite Newton-Schulz iterations,~\citet{Qian-2026-Convergence} studied inexact Muon in the degenerate spectral regime under generalized smoothness condition, and~\citet{Zhang-2026-Scale} established optimal rates for stochastic scale-invariant methods in operator-norm geometry and implemented the resulting matrix updates using inexact polar computations. Complementary analyses examine spectral updates through one-step progress, local curvature models, and non-Euclidean trust regions~\citep{Davis-2025-Spectral, Su-2025-Isotropic, Kovalev-2025-Understanding}. Alongside this theory, a rapidly growing literature develops practical Muon variants and scalable implementations, including large-scale Muon and MuonClip, NorMuon and Muon$^2$, adaptive and Gram-based orthogonalization and full-stack distributed systems~\citep{Liu-2025-Muon, Team-2025-Kimi, Li-2026-Normuon,Liu-2026-Muon2, Zhuang-2026-Amo, Zhang-2026-Gram, Khona-2026-Soap, Amsel-2026-Dion3}. Existing convergence theory remains tied to smoothness or generalized smoothness, whereas we ask what remains true for locally Lipschitz objectives.

The framework used here belongs to a line of work on set-valued first-order dynamics. Differential inclusion methods connect stochastic approximation with continuous-time limits~\citep{Benaim-2005-Stochastic} and stochastic subgradient methods converge on tame objectives under some regularity conditions~\citep{Davis-2020-Stochastic}. For automatic differentiation,~\citet{Bolte-2021-Conservative} has introduced conservative fields and developed their calculus, while~\citet{Lewis-2021-Structure} characterized their specific structure on semialgebraic functions through Whitney stratifications, building on earlier variational analysis of stratifiable functions~\citep{Bolte-2007-Clarke}. Related tools have since been applied to adaptive first-order methods~\citep{Xiao-2024-Adam, Hu-2025-Learning, Ding-2025-Stochastic, Ding-2025-Adam}, while~\citet{Jiang-2026-Adaptive} obtained the finite-time guarantees for nonconvex nonsmooth optimization using matrix online learning methods through smoothing. Our analysis instead applies conservative field dynamics directly to orthogonalized momentum for nonsmooth optimization.

Beyond the aforementioned classical convex lower bounds, the information-theoretic arguments give minimax bounds for stochastic first-order oracles, while local oracle constructions extend such results to randomized algorithms and general normed spaces~\citep{Agarwal-2012-Information, Braun-2017-Lower, Braun-2024-Corrections}. In smooth nonconvex settings, \citet{Carmon-2020-Lower, Carmon-2021-Lower} established deterministic oracle lower bounds for finding stationary points, and \citet{Arjevani-2023-Lower} developed corresponding stochastic lower bounds. In nonsmooth nonconvex settings, the notion of Clarke stationarity~\citep{Clarke-1990-Optimization} yields impossibility and lower bound results for first-order methods~\citep{Zhang-2020-Complexity, Kornowski-2022-Oracle}. Our lower bound is derived for convex first-order optimization in matrix-norm geometry. Indeed, we adapt hidden-direction arguments to operator-norm geometry using nearly orthogonal rank-one matrices and operator-nuclear norm duality.

Finally, the diagonal reduction of polar updates connects Muon with normalized and sign methods, which likewise discard magnitude information~\citep{Balles-2018-Dissecting, Bernstein-2018-Signsgd, Cutkosky-2020-Momentum}. In this context,~\citet{Karimireddy-2019-Error} has shown that sign subgradient descent can fail on a convex Lipschitz objective and that error feedback can restore convergence. Error feedback has since been developed for biased compression and for Muon and related methods~\citep{Beznosikov-2023-Biased,Demidovich-2023-Guide, Gruntkowska-2026-Error, Parshakova-2026-Muon}. Our MAGD method follows a different route by leaving the inexact Muon direction unchanged and controlling its influence through a separate gradient branch, yielding an improved performance across different tasks. 

\paragraph{Organization.} The remainder is organized as follows. In Section~\ref{sec:prelim}, we review conservative fields, Muon, and black-box first-order oracle model. In Section~\ref{sec:results}, we present the counterexample, establish asymptotic convergence for inexact Muon under increasing momentum, introduce MAGD, and prove the matching upper and lower bounds. In Section~\ref{sec:exp}, we conduct the experiments on synthetic and real data and the numerical results show the efficiency of our method. We conclude this paper in Section~\ref{sec:conclu}.

\section{Preliminaries} \label{sec:prelim}
We first review tools from nonsmooth analysis and state the assumptions used throughout the paper. Then, we introduce the inexact Muon and its operator-norm geometry, followed by the first-order oracle model used in our subsequent analysis.

\subsection{Nonsmooth analysis and conservative fields}
We consider the matrix optimization problem given by 
\begin{equation*}
\min_{X \in \br^{m \times n}} f(X). 
\end{equation*}
Throughout, we assume that $f$ is \emph{locally Lipschitz}, meaning that it is Lipschitz on a neighborhood of every point in $\br^{m\times n}$. Rademacher's theorem guarantees that such a function is differentiable almost everywhere and the Clarke subdifferential of $f$ \citep{Clarke-1990-Optimization} at $W \in \br^{m \times n}$ is defined by
\begin{equation*}
\partial f(W) := \conv\left\{V\in\br^{m\times n}:\exists W_k\to W\text{ such that $f$ is differentiable at every $W_k$ and }\grad f(W_k)\to V\right\},
\end{equation*}
where $\conv(A)$ denotes the convex hull of $A$ for any $A \subseteq \mathbb R^{m\times n}$. However, backpropagation might not output an element of the Clarke subdifferential during neural network training \cite[Remark 13]{Bolte-2021-Conservative}. To capture what backpropagation actually returns, a more general model of calculus, called \emph{conservative field}, was introduced in \citep{Bolte-2021-Conservative}. In particular, a conservative field of $f$ maps each $W\in\br^{m\times n}$ to a set of potential differentials of $f$ at $W$, and generalizes $\partial f$ to accommodate the operations and compositions that are standard in neural networks. It becomes almost everywhere the singleton consisting of the gradient of $f$ \cite[Theorem 1]{Bolte-2021-Conservative}.

We abbreviate any set-valued mapping $D:A \to 2^B$ by $D:A \rightrightarrows B$, where $A,B$ are sets and $2^B$ is the power set of $B$. The graph of $D$ is defined as $\{(x,z): x\in A, z\in D(x)\}$. Let $I \subseteq \br$ be a closed and possibly unbounded interval. A function $x:I \to \br^{m \times n}$ is \emph{absolutely continuous} if it is differentiable almost everywhere and satisfies $x(t_2)-x(t_1)=\int_{t_1}^{t_2} \dot x(s)ds$ for all $t_1,t_2\in I$. Formally, we have
\begin{definition}
Let $f \colon \br^{m \times n} \to \br$ be locally Lipschitz, and let $\DCal_f \colon \br^{m\times n} \rightrightarrows \br^{m\times n}$ be a set-valued mapping with closed graph and nonempty and compact values. We say that $\DCal_f$ is a \emph{conservative field} of $f$ if, for every absolutely continuous $\gamma \colon [0,1] \to \br^{m\times n}$ and every measurable $v \colon [0,1] \to \br^{m\times n}$ with $v(t) \in \DCal_f(\gamma(t))$ for almost every $t \in [0,1]$, the following statement holds, 
\begin{equation*}
f(\gamma(1)) - f(\gamma(0)) = \int_0^1 \big\langle \dot{\gamma}(t), v(t) \big\rangle  dt .
\end{equation*}
\end{definition}
For an objective $f$ represented by a finite computational graph whose elementary functions admit conservative fields, the possible outputs of backpropagation form a conservative field for $f$~\citep[Theorem~8]{Bolte-2021-Conservative}. We impose the following assumption on $f$ and its conservative field $\mathcal D_f$, which has become recognized as standard in the literature~\citep{Bolte-2021-Conservative, Davis-2020-Stochastic, Ding-2025-Adam}. 
\begin{assumption}\label{assumption:muon-data}
The objective function $f$ admits a convex-valued conservative field $\DCal_f:\br^{m \times n} \rightrightarrows \br^{m\times n}$ for which the weak Sard condition holds. Equivalently, we have (i) $f$ is locally Lipschitz, (ii) $\DCal_f$ is a convex-valued conservative field for $f$, and (iii) $\{f(W):0 \in \DCal_f(W)\}$ has empty interior in $\br$.
\end{assumption}
The objective functions obtained from standard architectures, including convolutional networks~\citep[Chapter~9]{Goodfellow-2016-Deep} and transformers~\citep{Vaswani-2017-Attention}, satisfy Assumption~\ref{assumption:muon-data}. Indeed, the automatic differentiation construction of~\citet[Section~5.2 and Corollary~6]{Bolte-2021-Conservative} yields the conservative field containing the outputs of backpropagation. Taking its convex hull can preserve conservativity and does not exclude any such output~\citep[Remark~3(e)]{Bolte-2021-Conservative}, while the nonsmooth Morse-Sard theorem gives the weak Sard condition~\citep[Theorem~5]{Bolte-2021-Conservative}. To obtain finite-time guarantees for our algorithm, we impose the additional structure. 
\begin{assumption}
\label{assumption:convex}
The objective function $f$ is convex, and $\mathcal D_f = \partial f$.
\end{assumption}
For convex $f$, the Clarke subdifferential coincides with the convex subdifferential, so $\partial f$ can be understood in either sense~\citep[Proposition~2.2.7]{Clarke-1990-Optimization} without abuse of notation. 

\subsection{Muon and operator-norm geometry}
We study Muon in the operator-norm geometry. For $A \in \br^{m \times n}$, let $\|A\|_{\rm op} := \max_{\|x\|_2\leq1}\|Ax\|_2$, where $\|\cdot\|_2$ denotes the Euclidean norm, and let $\|A\|_{\rm nuc}$ denote the nuclear norm. These norms are dual under the Frobenius inner product. 

Newton-Schulz~\citep{Jordan-2024-Muon} and PolarExpress~\citep{Amsel-2026-Polar} are standard subroutines for Muon to approximate orthogonalization. At each iteration of inexact Muon, these methods start from a normalized momentum matrix and apply a finite number of polynomial iterations. Given that $M$ is a momentum matrix and $\sigma \geq 0$ is a parameter in normalization, the resulting approximation has the following form of 
\begin{equation*}
\widetilde{M} \textnormal{poly}(\widetilde{M}^\top \widetilde{M}),\quad \textnormal{where } \widetilde{M} = \begin{cases}
\tfrac{M}{\|M\|_F+\sigma}, & M \not= 0,\\ 0, & M=0, 
\end{cases}
\end{equation*}
where $\|\cdot\|_F$ denotes the Frobenius norm and $\textnormal{poly}(\cdot)$ is the scalar polynomial determined by the chosen method and number of polynomial iterations. These subroutines preserve the nonzero singular subspaces and modify only the singular values. Motivated by these properties, we use the following abstract model. For $\delta \in [0,1)$, $Z \in \br^{m \times n}$ with $\|Z\|_F \leq 1$ and $r:=\operatorname{rank}(Z)$, we define $\operatorname{polar}_{\delta}(0):=\{0\}$ and, for $Z\ne0$,
\begin{equation} \label{eq:inexact-polar-envelope}
\operatorname{polar}_\delta(Z) :=
\left\{
U\operatorname{Diag}(\phi_1,\ldots,\phi_{r})V^\top
\ \middle| \
\begin{array}{l}
Z = U\operatorname{Diag}(\sigma_1,\dots,\sigma_{r}) V^\top
\text{ is a compact SVD},\\
0 \leq \phi_j\leq 1+\delta \quad \textnormal{for every } j, \\
\phi_j \geq 1-\delta \quad \textnormal{whenever } \sigma_j \geq \delta
\end{array}
\right\}.
\end{equation}
Given a sequence $\{\delta_t\}_{t \geq 0} \subseteq [0,1)$ and $\{\sigma_t\}_{t \geq 0} \subseteq [0, +\infty)$, the inexact
Muon recursion is
\begin{equation}
\label{eq:muon}
G_t \in \DCal_f(W_t),\quad M_{t+1}=\beta_tM_t+(1-\beta_t)G_t,\quad \widetilde M_{t+1}=\tfrac{M_{t+1}}{\|M_{t+1}\|_F+\sigma_t}, \quad W_{t+1} \in W_t - \eta_t \operatorname{polar}_{\delta_t} (\widetilde M_{t+1})
\end{equation}
with the convention of $\widetilde M_{t+1}=0$ when $M_{t+1}=0$ and $\sigma_t=0$. We note that $\operatorname{polar}_0(Z)$ is a singleton set for any $Z \in \br^{m \times n}$ such that $\|Z\|_F\leq 1$~\cite[Theorem 8.1]{Higham-2008-Functions}. Therefore, if $\delta_t \equiv 0$ and $M_{t+1} = U_{t+1}\Sigma_{t+1} V_{t+1}^\top$ is a compact SVD, the inexact Muon scheme~\eqref{eq:muon} reduces to the exact Muon recursion
\begin{align*}
W_{t+1} \in W_t-\eta_t\operatorname{polar}_{0}\left(\tfrac{M_{t+1}}{\|M_{t+1}\|_F+\sigma_t}\right),\quad \textnormal{equivalently},\quad W_{t+1}= W_t - \eta_t U_{t+1}V_{t+1}^\top.
\end{align*}
The exact Muon recursion thus minimizes a linear model over an operator-norm ball centered at $W_t$ \citep{Pethick-2025-Training}, given by, $W_{t+1}\in\argmin_{\|W-W_t\|_{\rm op}\leq \eta_t} \langle W-W_t, M_{t+1}\rangle$.

\subsection{Oracle complexity of convex nonsmooth optimization} \label{subsec:oracle-model}
We use the black-box first-order oracle model~\citep{Nemirovski-1983-Problem, Braun-2017-Lower} to measure the number of oracle calls needed to optimize in the operator-norm geometry, where the unknown objective $f$ belongs to a known class $\FCal$ of convex Lipschitz functions on $\br^{m \times n}$, and accuracy is measured relative to its minimum over a nonempty compact convex set $\XCal$. At each iteration, an algorithm queries a point using the preceding oracle replies, and the oracle returns the function value and a subgradient at that point. Given $\epsilon>0$, the algorithm seeks to output $\widehat{W}$ satisfying $f(\widehat{W}) \leq \min_{W \in \XCal} f(W)+\epsilon$. This paradigm includes projected subgradient descent and mirror descent~\citep{Beck-2017-First}, conditional gradient~\citep{Jaggi-2013-Revisiting} and Muon~\citep{Jordan-2024-Muon}. 

More specifically, given $\FCal$ and $\XCal$, an oracle $\mathsf{O}$ supplies partial information $\mathsf{O}_f(W)$ about an unknown instance $f\in\mathcal F$ at a queried point $W\in\br^{m\times n}$. The algorithm may query and report anywhere in $\br^{m\times n}$. We consider a \emph{first-order oracle} $\mathsf O$, which answers a query $W \in \br^{m \times n}$ by
\begin{equation*}
\mathsf{O}: \FCal \times \br^{m \times n} \to \br \times \br^{m \times n},\quad \mathsf{O}_f(W)=(f(W),g),\quad g\in\partial f(W),
\end{equation*}
where $f(W)$ is the function value and $g$ is a subgradient of $f$ at $W$. We also require $\mathsf{O}$ to be \emph{local}: if $f_1,f_2 \in \FCal$ agree on a neighborhood of $W$, we have $\mathsf{O}_{f_1}(W)=\mathsf{O}_{f_2}(W)$. Thus, the oracle reply at $W$ depends only on the behavior of $f$ near $W$.

Let $\ACal(\mathsf{O})$ denote the set of deterministic algorithms based on $\mathsf{O}$. Each algorithm proceeds by maintaining $\Pi_q:=\left((X^1, \mathsf{O}_f(X^1)), \ldots, (X^q,\mathsf{O}_f(X^q))\right)$ for $q \geq 0$ (with $ \Pi_0:=\emptyset$) and applying a decision rule that either outputs a point $\widehat{X} \in \br^{m \times n}$ or continues to select the next query $X^{q+1} \in \br^{m \times n}$. Given $\mathsf{A} \in \ACal(\mathsf{O})$, $f \in \FCal$ and a tolerance $\epsilon>0$, we write $\mathsf {T}_{\mathsf{A}}(f, \epsilon)$ for the number of queries that $\mathsf{A}$ makes to reach an $\epsilon$-optimal solution. Equivalently, $\mathsf{T}_{\mathsf{A}}(f,\epsilon)=q$ if $\mathsf{A}$ stops after $q$ completed calls with an output satisfying $f(\widehat{X})-\min_{X \in \XCal} f(X) \leq \epsilon$, and $\mathsf{T}_{\mathsf{A}}(f,\epsilon)=+\infty$ otherwise. The \emph{worst-case} oracle complexity is defined by
\begin{equation} \label{eq:oracle-worst-case-complexity}
\mathfrak{m}_{\rm wcc}(\XCal,\FCal,\mathsf{O},\epsilon):=\inf_{\mathsf{A} \in \ACal(\mathsf{O})} \sup_{f \in \FCal} \mathsf{T}_{\mathsf{A}}(f,\epsilon).
\end{equation}
Let $\Delta(\FCal)$ denote the set of probability distributions on $\mathcal F$, we define the \emph{distributional} complexity as
\begin{equation} \label{eq:oracle-distributional-complexity}
\mathfrak{m}_{\rm dc}(\XCal,\FCal,\mathsf{O},\epsilon):=\sup_{\PCal \in \Delta(\mathcal F)}\ \inf_{\mathsf{A} \in \ACal(\mathsf{O})} \EE_{f \sim \PCal}[\mathsf{T}_{\mathsf{A}}(f,\epsilon)].
\end{equation}
Bounding this quantity yields stronger lower bounds. Indeed, Yao's minimax principle~\citep{Yao-1977-Probabilistic} yields 
\begin{equation} \label{eq:oracle-complexity-order}
\mathfrak{m}_{\rm dc}(\XCal,\FCal,\mathsf{O},\epsilon) \leq \mathfrak{m}_{\rm wcc}(\XCal,\FCal,\mathsf{O},\epsilon).
\end{equation}
which implies that a lower bound for $\mathfrak m_{\rm dc}$ is also the lower bound for $\mathfrak{m}_{\rm wcc}$.

\section{Main Results}\label{sec:results}
We first show that Muon fails to converge to a global optimal solution for some convex and Lipschitz functions when the momentum factor is constant. Then, we show that inexact Muon with a particular momentum schedule and vanishing polar error converges asymptotically in general nonconvex settings. We also propose a new algorithm based on Muon and proves that it achieves the asymptotic convergence in nonconvex settings and the optimal convergence rate in convex settings.

\subsection{Convergence behavior of Muon}
\citet[Counterexample~3.3]{Parshakova-2026-Muon} provides a convex and Lipschitz function for which Muon with $\beta_t \equiv \beta \in [0,1)$ and a nonincreasing step size sequence fails to converge from some starting points, and \citet[Counterexample~3.4]{Parshakova-2026-Muon} strengthens this conclusion to almost every starting point and adaptive step sizes based on the full gradient history, but only for $\beta_t \equiv \beta \in [0,\frac{1}{2})$. Our first result closes the gap by proving this almost everywhere failure for every $\beta \in [0,1)$. 
\begin{theorem}\label{thm:fixed-momentum-failure}
For any fixed $m,n \geq 2$ and $\beta \in [0,1)$, there exists a convex and Lipschitz function $f: \br^{m\times n} \to \br$ such that the iterates generated by Muon in Eq.~\eqref{eq:muon} with $\delta_t\equiv0$, $\DCal_f = \partial f$, $\beta_t \equiv \beta$, $M_0 = G_0$, and arbitrary adaptive step sizes $\{\eta_t\}_{t\ge 0}$ based on the full subgradient history satisfy
\begin{equation*}
\inf_{t \geq 0} f(W_t) > f^\star := \inf_{W \in \br^{m\times n}} f(W).
\end{equation*}
for almost every $W_0 \in \br^{m \times n}$,
\end{theorem}
\begin{proof}
For any matrix $X$, we denote $[X]_{ij}$ as its $(i,j)$-th entry. We define 
\begin{equation*}
f(W) = A_1|[W]_{11}+[W]_{22}|+A_2\max\{[W]_{11}-[W]_{22},0\}+\max\{[W]_{22}-[W]_{11},0\}
\end{equation*}
where $A_1,A_2$ are positive constants and depend only on $\beta$. For simplicity, we denote the set-valued sign function by $\operatorname{sgn}:\br \rightrightarrows \br$ as 
\begin{equation*}
\operatorname{sgn}(x) := \begin{cases}
\{-1\}, & \textnormal{if } x<0,\\
[-1,1], & \textnormal{if } x=0,\\
\{1\}, & \textnormal{otherwise}.
\end{cases}
\end{equation*}
\textbf{Step 1:} We can rewrite Muon's updates in the coordinate system of $[W]_{11}+[W]_{22}$ and $[W]_{11}-[W]_{22}$ using the form of $f$. Indeed, we define $p_t=[W_t]_{11}+[W_t]_{22}$, $q_t=[W_t]_{11}-[W_t]_{22}$, $\alpha_t=\frac{1}{2}([M_t]_{11}+[M_t]_{22})$ and $\zeta_t=\frac{1}{2}([M_t]_{11}-[M_t]_{22})$. Since $f$ depends only on $[W]_{11}$ and $[W]_{22}$, both subgradient and momentum are supported only on the diagonal entries as follows, 
\begin{equation*}
[M_{t+1}]_{11} = (\alpha_{t+1}+\zeta_{t+1}),\quad [M_{t+1}]_{22}=(\alpha_{t+1}-\zeta_{t+1}), \quad [M_{t+1}]_{ij}=0 \textnormal{ if } i \neq j. 
\end{equation*}
Since $M_{t+1}=U_{t+1}\Sigma_{t+1}V_{t+1}^\top$ is a compact SVD, $P_{t+1}=U_{t+1}V_{t+1}^\top$ depends on the same two entries and all other entries are $0$. We write $u_{t+1}:=[P_{t+1}]_{11}$ and $v_{t+1}:=[P_{t+1}]_{22}$. By definition, we have
\begin{equation*}
\begin{array}{ll}
u_{t+1} \in \operatorname{sgn}([M_{t+1}]_{11})=\operatorname{sgn}(\alpha_{t+1}+\zeta_{t+1}), & p_{t+1}=p_t-\eta_t(u_{t+1}+v_{t+1}), \\
v_{t+1} \in \operatorname{sgn}([M_{t+1}]_{22}) = \operatorname{sgn}(\alpha_{t+1}-\zeta_{t+1}), & q_{t+1}=q_t-\eta_t(u_{t+1}-v_{t+1}).
\end{array}
\end{equation*}
In addition, we have $\alpha_{t+1}=\beta\alpha_t+(1-\beta)a_t$ and $\zeta_{t+1}=\beta\zeta_t+(1-\beta)b_t$ where 
\begin{equation*}
a_t\in A_1\operatorname{sgn}(p_t), \quad b_t \in \begin{cases} \{A_2\}, & q_t > 0, \\
\{-1\}, & q_t < 0, \\ 
[-1, A_2], & q_t = 0. \end{cases}
\end{equation*}
If $|\zeta_{t+1}|>|\alpha_{t+1}|$, we have $u_{t+1}$ and $v_{t+1}$ have opposite signs and $\frac{1}{2}(u_{t+1}-v_{t+1}) \in \operatorname{sgn}(\zeta_{t+1})$. Thus, we have
\begin{equation} \label{eq:Muon-nonconvergence-simplified-update}
p_{t+1}=p_t,\quad q_{t+1}=q_t-2\eta_td_t, \quad \textnormal{where} \quad d_t\in \operatorname{sgn}(\zeta_{t+1}).
\end{equation}
\textbf{Step 2:} We define $T_0:=\inf\{t \geq 0: q_t=0\}$ and show that there are $A_1,A_2>0$ such that $|\zeta_{t+1}|>|\alpha_{t+1}|$ for all $t<T_0$. Note that we will show in \textbf{Step 3} that $T_0=+\infty$ for almost every $W_0 \in \br^{m \times n}$. 

We consider the case of $\beta=0$. Indeed, we set $A_1=\frac{1}{4}$ and $A_2=1$. For every $t<T_0$, we have $q_t\neq 0$. This implies $\zeta_{t+1}=b_t\in\{-1,1\}$ and $|\alpha_{t+1}|=|a_t|\le A_1$. Hence, $|\zeta_{t+1}| = 1 > |\alpha_{t+1}|$. 

We proceed to the general case of $\beta \in (0,1)$. Here, the construction of $A_1$ and $A_2$ is \textit{implicit}. Indeed, we define $g: (1,\frac{1}{\beta}) \to \br$ by $g(x):= \inf_{i \geq 0,j \geq 1} |\beta^i-x (1-\beta^j)|$ and show that there exists $\mu \in (1,\frac{1}{\beta})$ such that $g(\mu)>0$. Indeed, for any $x \in (1,\frac{1}{\beta})$, we define
\begin{equation*}
\ACal = \{\beta^i: i \geq 0\} \cup \{0\}, \quad \BCal_x:=\{x(1-\beta^j):j \geq 1\} \cup \{x\},\quad \BCal = \left\{\tfrac{\beta^i}{1-\beta^j}: i \geq 0, j \geq 1\right\}. 
\end{equation*}
Both $\ACal$ and $\BCal_x$ are compact. Since $x>1$, we have $0 \notin \BCal_x$ and $x \notin \ACal$. Thus, if $x \notin \BCal$, then $\ACal \cap \BCal_x = \varnothing$ and $g(x)=\operatorname{dist}(\ACal,\BCal_x)>0$. Since $\BCal$ is countable, we can choose $\mu \in (1,\frac{1}{\beta}) \setminus \BCal$, which gives $g(\mu)>0$. Then, we set $K = \max\{\lceil\log_{1/\beta}(2/\delta)\rceil-1, \lceil\log_{1/\beta}(2/(1-\beta))\rceil\}$ where $\delta=g(\mu)$, and define 
\begin{equation*}
A_1=\tfrac{1}{2}\min\left\{\tfrac{1}{2},\tfrac{\delta}{2},1-\beta\mu\right\}>0, \quad A_2=\tfrac{\mu}{\beta^K}-1>0. 
\end{equation*}
Then, we show $|\zeta_{t+1}|>|\alpha_{t+1}|$ for all $t<T_0$. By using the definition of $u_{t+1}$ and $v_{t+1}$ and $\operatorname{sgn}=\partial|\cdot|$, we have $0 \leq (u_{t+1}-v_{t+1})((\alpha_{t+1}+\zeta_{t+1})-(\alpha_{t+1}-\zeta_{t+1})) = 2\zeta_{t+1}(u_{t+1}-v_{t+1})$. In each of the cases $\zeta_{t+1}>0$, $\zeta_{t+1}=0$, and $\zeta_{t+1}<0$, the $q$-update can be written as
\begin{equation}\label{eq:general-q-update}
q_{t+1}=q_t-\lambda_t s_{t+1} \textnormal{ for } \lambda_t \in [0, 2\eta_t] \textnormal{ and } s_{t+1} \in \operatorname{sgn}(\zeta_{t+1}).
\end{equation}
Since $M_0=G_0$, we have $|\alpha_0|=|a_0| \leq A_1$ and $\zeta_0=b_0\in[-1,A_2]$. By induction, we have $|\alpha_t| \leq A_1$ and $\zeta_t\in[-1,A_2]$ for all $t \geq 0$. Thus, it remains to show that $|\zeta_{t+1}|\ge2A_1$ for all $t<T_0$.

Before $T_0$, we call a maximal consecutive block of indices on which $q_t<0$ a negative run, and define a positive run analogously. Fix $t<T_0$. If $q_t>0$, then $t$ is in a positive run starting at some $t' \leq t$ and
\begin{equation*}
\zeta_{t+1}=(1-\beta^{t+1-t'})A_2+\beta^{t+1-t'}\zeta_{t'} \geq (1-\beta^{t+1-t'})A_2-\beta^{t+1-t'} \geq (1-\beta)A_2-\beta. 
\end{equation*}
Since $K \geq \lceil\log_{1/\beta}(2/(1-\beta))\rceil$, we have $\beta^K \leq \frac{1-\beta}{2}$. By definition of $A_2$, we have $(1-\beta)A_2-\beta>1\geq 2A_1$. Thus, we have $\zeta_{t+1} \geq 2A_1$. Then, it remains to show $q_t<0$ implies $|\zeta_{t+1}|\geq 2A_1$. 

We let $t' \leq t$ start the negative run containing $t$ and define
\begin{equation*}
\ECal:=\{-1,A_2\} \cup \bigcup_{r \geq 1}I_r, \textnormal{ where } I_r:=[(1-\beta^r)A_2-\beta^r,(1-\beta^r)A_2] \textnormal{ for all } r \geq 1.
\end{equation*}
We claim that $\zeta_{t'} \in \ECal$. Indeed, we first consider $t'=0$ and obtain that $q_0<0$ and $\zeta_{t'}=b_0=-1$. Furthermore, we consider $t'>0$ and and the preceding positive run starts at $t=0$. This implies $\zeta_{i}=A_2$ for all $i\leq t'$. Thus, we have $\zeta_{t'}=A_2$. Finally, we consider $t'>0$ and the preceding positive run starts at $t=\ell>0$. This implies $q_{\ell-1}<0<q_\ell$, which together with Eq.~\eqref{eq:general-q-update} yields $\lambda_{\ell-1}>0$ and $s_\ell<0$. Since $s_\ell \in \operatorname{sgn}(\zeta_\ell)$, we have $\zeta_\ell \leq 0$. Thus, we have $q_\ell, \ldots, q_{t'-1}>0$ and $q_{t'}<0$. This implies $\zeta_{t'}=(1-\beta^r)A_2+\beta^r\zeta_\ell \in I_r$ where $r=t'-\ell \geq 1$. Therefore, we have  
\begin{equation*}
|\zeta_{t+1}|=|\beta^k\zeta_{t'}-(1-\beta^k)| \geq \beta^k\min \left\{\inf_{r \geq 1} \operatorname{dist}(I_r, \tfrac{1}{\beta^k}-1), |-1-(\tfrac{1}{\beta^k}-1)|, |A_2-(\tfrac{1}{\beta^k}-1)|\right\}.
\end{equation*}
where $k=t+1-t'$ is the length of the run up to $t$. 

It remains to prove $D_0^{(k)}=|1-(1+A_2)\beta^k| \geq 2A_1$ and $D_r^{(k)}=\beta^k\operatorname{dist}(I_r, \frac{1}{\beta^k}-1)\geq 2A_1$ for all $r\geq 1$. Indeed, we consider the case of $k \leq K$. Then, we have
\begin{equation*}
\inf_{r\geq 1} D_r^{(k)} = \inf_{r\geq 1} \inf_{s\in[0,\beta^r]} |1-(1-\beta^r)\beta^k (A_2+1)-s\beta^k| \geq \inf_{r\geq 1} \ |1-(1-\beta^r)\beta^k(A_2+1)|-\beta^{k+r}.
\end{equation*}
For all $r \geq 1$, we have $\beta^{k+r} \leq \beta^{k+1} \leq \frac{\delta}{2}\beta^{k-K}$. By the definition of $K$, we have
\begin{equation*}
|1-(1-\beta^r)\beta^k(A_2+1)| = |1-(1-\beta^r)\mu\beta^{k-K}| \geq \beta^{k-K}\delta.
\end{equation*}
Thus, we have $\inf_{r\geq 1} D_r^{(k)} \geq \frac{\delta}{2} \geq 2A_1$ and 
\begin{equation*}
D_0^{(k)}=|1-\mu\beta^{k-K}|\geq (\mu-1)\beta^{k-K} \geq \mu-1 \geq \delta \geq 2A_1. 
\end{equation*}
Then, we consider the case of $k>K$. Since $\frac{1}{\beta^k}-1 \geq \frac{1}{\beta^{K+1}}-1>\frac{\mu}{\beta^K}-1=A_2$, we have
\begin{equation*}
\inf_{r\geq 1} D_r^{(k)}\geq D_0^{(k)} = |1-(1+A_2)\beta^k| = 1-\mu\beta^{k-K} \geq 1-\mu\beta \geq 2A_1.
\end{equation*}
Since $|\alpha_{t+1}|\le A_1<2A_1\le|\zeta_{t+1}|$, Eq.~\eqref{eq:Muon-nonconvergence-simplified-update} holds for all $t<T_0$. 

\noindent \textbf{Step 3:} We show that $T_0=+\infty$ for almost every $W_0$. Indeed, we let $\GCal_\beta$ be the set of subgradients of $f$ at points $\{W:|[W]_{11}|\not=|[W]_{22}|\}$ and $\eta_t=\phi_t(G_0,\ldots,G_t)$ where $\phi_t(\cdot)$ is the adaptive step-size rule. Then, we define
\begin{equation*}
\HCal_\phi:=\left\{2\sum_{t=0}^{T-1}\phi_t(H_0,\ldots,H_t)\varepsilon_t: T \geq 0,H_t \in \GCal_\beta, \varepsilon_t \in \{\pm 1\} \right\}.
\end{equation*}
Thanks to the particular form of $f$, the set $\GCal_\beta$ is finite even though $\{W:|[W]_{11}|\not=|[W]_{22}|\}$ is infinite. Since there are finitely many choices for each $T$, the set $\HCal_\phi$ is countable. Note that $\{W_0: q_0 \in \HCal_\phi\}$ is a countable union of affine hyperplanes. Thus, its union with $\{p_0q_0=0\}$ is a zero-measure set. If $q_0 \notin \HCal_\phi$, $p_0q_0 \neq 0$, and $T_0<+\infty$, Eq.~\eqref{eq:Muon-nonconvergence-simplified-update} before $T_0$ gives $p_t=p_0\neq 0$. Thus, all subgradients before $T_0$ belong to $\GCal_\beta$ and $q_0 \in \HCal_\phi$. This yields a contradiction. Thus, $T_0=+\infty$ and $p_t=p_0$ for every $t$ for almost every $W_0$, which implies $f(W_t) \geq A_1|p_0|>0 = \inf_{W \in \br^{m \times n}} f(W)$. 
This completes the proof. 
\end{proof}
We remark that the failure of Muon in Theorem~\ref{thm:fixed-momentum-failure} is not due to unbounded iterates. Indeed, the proof has shown that the iterates $\{W_t\}_{t \geq 0}$ remain bounded whenever the step sizes $\{\eta_t\}_{t \geq 0}$ are bounded. Our second result further shows that this failure stems from fixed momentum rather than orthogonalization.
\begin{theorem}
\label{thm:muon-asymptotic-convergence}
Suppose that Assumption~\ref{assumption:muon-data} holds, $\tau>0$ and $\{(W_t,M_t)\}_{t \geq 0}$ is generated by Eq.~\eqref{eq:muon} with $\beta_t=1-\tau \eta_t$, $\eta_t>0$, $\eta_t\to0$, $\sum_{t=0}^{\infty}\eta_t=+\infty$, $\delta_t\to0$, and $\sup_{t\geq 0}\sigma_t<+\infty$. If $\sup_t \|W_t\|_F<+\infty$, every cluster point of $\{W_t\}_{t \geq 0}$ is a $\mathcal D_f$-stationary point, $M_t \to 0$, and the sequence $\{f(W_t)\}_{t \geq 0}$ converges. In addition, if Assumption~\ref{assumption:convex} holds, $f(W_t) \to \inf_W f(W)$ as $t \to +\infty$. 
\end{theorem}
We first define the $\varepsilon$-graph enlargement of the mapping $D:\br^d \rightrightarrows \br^d$ as $D^\varepsilon(x) := \{v: \exists y \textnormal{ with } \|y-x\|_{2} \leq \varepsilon \textnormal{ and } \inf_{z \in D(y)} \|v-z\|_{2} \leq \varepsilon\}$. Here, $D$ is locally bounded if for any $x\in\br^{d}$ there exists a neighborhood $V_x \subseteq \br^d$ such that $\cup_{y \in V_x} D(y)\subseteq \br^d$ is bounded. The typical examples of locally bounded mappings are conservative fields~\citep[Remark~3]{Bolte-2021-Conservative}. We then provide a general sufficient condition for the asymptotic convergence of updates that approximate a differential inclusion, which is the key to our subsequent analysis. For the sake of completeness, we provide the detailed proof. 
\begin{lemma} \label{lem:deterministic-asymptotic-convergence-principle}
Let $D:\br^d \rightrightarrows \br^d$ be locally bounded, has a closed graph, and have nonempty convex values. We consider the update $x_{t+1}=x_t+\eta_t d_t$ where $d_t\in D^{\varepsilon_t}(x_t)$, $\eta_t>0$, $\eta_t \to 0$, $\sum_{t=0}^{\infty}\eta_t=\infty$, $\varepsilon_t \geq 0$, $\varepsilon_t \to 0$ and $\sup_{t \geq 0}\|x_t\|_2<+\infty$. Let $B \subseteq \br^d$, $\varphi:\br^d \to \br$ be continuous and $L$ be the set of limit points of the sequence $\{x_t\}_{t \geq 0}$ as defined by 
\begin{equation} \label{eq:def-limit-set}
L := \bigcap_{N \geq 0} \overline{\{x_t:t \geq N\}}.
\end{equation}
Suppose that $\varphi(B\cap L)$ has an empty interior and $\varphi$ is a Lyapunov function for $D$ (every absolutely continuous solution $\gamma: [0,+\infty) \to L$ of $\dot\gamma(s)\in D(\gamma(s))$ for almost every $s \geq 0$ satisfies $\varphi(\gamma(s))\le\varphi(\gamma(0))$ for all $s \geq 0$, with $<$ for every $s>0$ whenever $\gamma(0) \notin B$). Then, we have $L \subseteq B$, $\varphi$ becomes a constant on $L$, and $\{\varphi(x_t)\}_{t \geq 0}$ converges. In addition, every cluster point of $\{x_t\}_{t \geq 0}$ lies in $B$.
\end{lemma}
\begin{proof}
To apply the arguments from~\citet{Benaim-2005-Stochastic}, we construct the continuous-time version of our updates. Let $\lambda_0=0$ and $\lambda_t=\sum_{k=0}^{t-1}\eta_k$ for $t\geq 1$. Since $\lambda_t\to+\infty$, we define $w:[0,+\infty)\to\br^d$ by
\begin{equation*}
w(\lambda_t+s)=x_t+s\tfrac{x_{t+1}-x_t}{\lambda_{t+1}-\lambda_t}=x_t+s d_t,\quad \textnormal{for } s \in [0,\eta_t) \textnormal{ and } t \geq 0.
\end{equation*}
Accordingly, we define $L_w:=\bigcap_{T \geq 0} \overline{\{w(s):s\geq T\}}$. We claim that $L=L_w$. Since $\|x_{t+1}-x_t\|_2 \leq \eta_t\|d_t\|_2$ and $\eta_t \to 0$, it suffices to show $\sup_{t \geq 0} \|d_t\|_2<+\infty$. Indeed, by the definition of $\varepsilon$-graph enlargement, there exist $\{(y_t, v_t)\}_{t \geq 0}$ such that $\|y_t-x_t\|_2 \leq \varepsilon_t$, $v_t \in D(y_t)$ and $\|v_t-d_t\|_2 \leq \varepsilon_t$. Since $\{x_t\}_{t\geq 0}$ is uniformly bounded and $\varepsilon_t \to 0$, $\{y_t\}_{t \geq 0}$ is uniformly bounded. It follows from the local boundedness of $D$ and $v_t \in D(y_t)$ that $\{v_t\}_{t \geq 0}$ is uniformly bounded, which implies $\sup_{t \geq 0} \|d_t\|_2<+\infty$.

The desired result holds if $L_w$ is internally chain transitive (ICT) for $D$. Indeed, since $L=L_w$ and $L_w$ is ICT,~\citet[Lemma~3.5 and Proposition~3.27]{Benaim-2005-Stochastic} and an empty interior of $\varphi(B\cap L)$ gives $L\subseteq B$ and $\varphi|_L=c$ for some constant $c$. If $\{\varphi(x_t)\}_{t \geq 0}$ does not converge to $c$, the boundedness of $\{x_t\}_{t \geq 0}$ guarantees a subsequence converging to some $\bar{x} \in L$ satisfying $\varphi(\bar{x}) \neq c$. This contradicts the continuity of $\varphi$. Since every cluster point of $\{x_t\}$ belongs to $L$, every cluster point lies in $B$.

To apply \cite[Theorem~3.6]{Benaim-2005-Stochastic} to show that $L_w$ is ICT, we show that $w$ is a bounded perturbed solution of some differential inclusion $\dot{x} \in \widetilde{D}(x)$ where $\widetilde{D}: \br^d \rightrightarrows \br^d$ is globally bounded, convex-valued, has a closed graph, and coincides with $D$ on a neighborhood of $\overline{\{x_t:t \geq 0\}}$. Indeed, we define $\widetilde{D}(x):=\{\chi(x)v: v\in D(x)\}$, where $\chi:\br^d \to [0,1]$ is a continuous compactly supported function that equals one on a neighborhood of $\overline{\{x_t,y_t:t \geq 0\}}$. Since $D$ is locally bounded and $\chi$ is supported on a compact set, $\widetilde{D}$ is globally bounded. $\widetilde{D}$ is convex-valued since $D$ is convex-valued. Moreover, we show that $\widetilde{D}$ has a closed graph. If $x_t \to x$ and $\chi(x_t)v_t \to z$ with $v_t \in D(x_t)$, the local boundedness of $D$ implies that $\{v_t\}_{t \geq 0}$ is bounded. Passing to a subsequence, $v_t \to v \in D(x)$ by the closed graph of $D$, and thus $z = \chi(x)v \in \widetilde{D}(x)$ using the continuity of $\chi$. We then show that $w(s)$ is a bounded perturbed solution of $\dot x \in \widetilde{D}(x)$. Indeed, for $s \in (\lambda_t,\lambda_{t+1})$, we have $\dot{w}(s)=d_t$ and $\|w(s)-x_t\|_2\leq\eta_t\|d_t\|_2$. This implies $\dot{w}(s) \in \widetilde{D}^{\varepsilon_t+\eta_t\|d_t\|_2}(w(s))$ where $\varepsilon_t+\eta_t\|d_t\|_2 \to 0$. In addition, $\{w(s)\}_{s \geq 0}$ is uniformly bounded since $\{x_t\}_{t\geq0}$ is uniformly bounded. 
\end{proof} 

\begin{proof}[Proof of Theorem~\ref{thm:muon-asymptotic-convergence}] By choosing $P_{t+1} \in \operatorname{polar}_{\delta_t}(M_{t+1}/(\|M_{t+1}\|_F+\sigma_t))$, we can rewrite Eq.~\eqref{eq:muon} as $(W_{t+1},M_{t+1}) = (W_{t},M_{t}) + \eta_t (-P_{t+1},\tau(G_t-M_t))$. Letting $Z_t=(W_t,M_t)$ and $D_t=(-P_{t+1},\tau(G_t-M_t))$, we have $Z_{t+1}=Z_t+\eta_t D_t$. Thus, we define the mapping $\FCal: \br^{m\times n} \times \br^{m\times n} \rightrightarrows \br^{m\times n}\times \br^{m\times n}$ as $\FCal(W,M):=\{(-P,\tau(G-M)):G \in \DCal_f(W),P \in \partial(\|\cdot\|_\textnormal{nuc})(M)\}$, where $\|\cdot\|_\textnormal{nuc}$ is the nuclear norm as a function $\br^{m\times n} \to \br$. Then, it suffices to show that the vectorized version of $Z_t$, $D_t$, and $\FCal$ satisfy the conditions of Lemma~\ref{lem:deterministic-asymptotic-convergence-principle}.

We first show that $\{Z_t\}_{t\geq 0}$ is bounded. Indeed, $\{W_t\}_{t\geq 0}$ is bounded. Since $\DCal_f$ is locally bounded, $\{G_t\}_{t\geq 0}$ is bounded. In addition, $M_{t+1}=(1-\tau\eta_t)M_t + \tau\eta_tG_t$ and $0 \leq \tau\eta_t \leq 1$ for all large $t$. Thus, $\{M_t\}$ is bounded. Putting these pieces together yields the desired result. 

We obtain that $\FCal$ satisfies the properties from Lemma~\ref{lem:deterministic-asymptotic-convergence-principle} since $\partial\operatorname{\|\cdot\|}$ and $\DCal_f$ are locally bounded, have closed graphs, and have nonempty convex values. To apply Lemma~\ref{lem:deterministic-asymptotic-convergence-principle}, we have to prove that (i) $D_t \in \FCal^{\varepsilon_t}(Z_t)$ with $\varepsilon_t \to 0$ and (ii) there exists a function $\varphi(Z)$ satisfying the desired properties. 

We show (i). Indeed, we construct $\widehat{Z}_t=(W_t,\widehat{M}_{t+1})$ and $\widehat{D}_t=(-\widehat{P}_{t+1},\tau(G_t-\widehat{M}_{t+1}))$ so that $\widehat{D}_t \in \FCal(\widehat{Z}_t)$, $\|\widehat{Z}_t-Z_t\|_2 \leq \varepsilon_t$, and $\|\widehat{D}_t-D_t\|_2 \leq \varepsilon_t$ with $\varepsilon_t \to 0$. If $M_{t+1}=0$, we set $\widehat{M}_{t+1}=\widehat{P}_{t+1}=0$. Otherwise, we let $k_{t+1}:=\operatorname{rank}(M_{t+1})$, $M_{t+1}=U\operatorname{Diag}(s_1,\ldots,s_{k_{t+1}})V^\top$ and $P_{t+1}=U\operatorname{Diag}(\phi_1,\ldots,\phi_{k_{t+1}})V^\top$ using SVD that outputs $P_{t+1}$ in Eq.~\eqref{eq:inexact-polar-envelope}. We define
\begin{equation*}
I_t := \{1\leq j\leq k_{t+1} :s_j \geq \delta_t(\|M_{t+1}\|_F+\sigma_t)\}
\end{equation*}
and 
\begin{equation*}
\begin{array}{ll}
\widehat{M}_{t+1} := U\operatorname{Diag}(\widehat{s}_1,\ldots,\widehat{s}_{k_{t+1}})V^\top, & \quad \widehat{s}_j = \begin{cases}
s_j, & j\in I_t,\\ 0, & j\notin I_t.\end{cases} \\
\widehat{P}_{t+1} := U\operatorname{Diag}(\widehat{\phi}_1,\ldots,\widehat{\phi}_{k_{t+1}})V^\top, & \quad \widehat{\phi}_j = \begin{cases}
1, & j\in I_t,\\ \min\{\phi_j,1\}, & j\notin I_t.\end{cases}
\end{array}
\end{equation*}
which implies $\|\widehat{M}_{t+1}-M_{t+1}\|_\textnormal{op} \leq \delta_t(\|M_{t+1}\|_F+\sigma_t)$ and $\|\widehat{P}_{t+1}-P_{t+1}\|_\textnormal{op} \leq \delta_t$. Here, we claim that $\widehat{P}_{t+1} \in \partial(\|\cdot\|_\textnormal{nuc})(\widehat{M}_{t+1})$. Indeed, since $\|\widehat{P}_{t+1}\|_\textnormal{op} \leq 1$ and $\langle \widehat{P}_{t+1},\widehat{M}_{t+1}\rangle = \|\widehat{M}_{t+1}\|_\textnormal{nuc}$, we have $\|Y\|_\textnormal{nuc} \geq \langle\widehat{P}_{t+1}, Y\rangle = \|\widehat{M}_{t+1}\|_\textnormal{nuc} + \langle\widehat{P}_{t+1},Y-\widehat{M}_{t+1}\rangle$ for all $Y \in \br^{m\times n}$. Therefore, we have $D_t \in \FCal^{\varepsilon_t}(Z_t)$, where 
$\varepsilon_t=\max\{\|\widehat{M}_{t+1}-M_{t+1}\|_F+\|M_{t+1}-M_t\|_F,\|\widehat{P}_{t+1}-P_{t+1}\|_F+\tau(\|\widehat{M}_{t+1}-M_{t+1}\|_F+\|M_{t+1}-M_t\|_F)\}$. Since $\{G_t\}_{t \geq 0}$ and $\{M_t\}_{t \geq 0}$ are bounded, $\|M_{t+1}-M_t\|_F=\tau\eta_t\|G_t-M_t\|_F\to0$. Since $\delta_t \to 0$ and $\sup_{t \geq 0}\sigma_t<+\infty$, we have $\|\widehat{M}_{t+1}-{M}_{t+1}\|_F \to 0,\|\widehat{P}_{t+1}-P_{t+1}\|_F \to 0$, which further implies $\varepsilon_t \to 0$. Putting these pieces together yields the desired result.  

We show (ii). Indeed, we define the function $\varphi(Z):=f(W)+\tau^{-1}\|M\|_\textnormal{nuc}$ and show that it satisfies the desired properties. We set $L$ as the limit set of $\{Z_t\}_{t \geq 0}$ in Eq.~\eqref{eq:def-limit-set} and $B:=\{Z:0 \in \DCal_f(W), M=0\}$. By Assumption~\ref{assumption:muon-data}, $\varphi(B\cap L) \subseteq \{f(W):0 \in \DCal_f(W)\}$ has the empty interior. Let $Z(s)=(W(s),M(s))$ be a solution of $\dot{Z} \in \FCal(Z)$ whose range lies in $L$. For $s$ almost everywhere, the matrices $P(s):=-\dot{W}(s)$ and $G(s):=M(s)+\tau^{-1}\dot{M}(s)$ satisfy $P(s) \in \partial(\|\cdot\|_\textnormal{nuc})(M(s))$ and $G(s) \in \DCal_f(W(s))$. By using the conservative field chain rule for $f$ and the chain rule for the nuclear norm, we have
\begin{equation*}
\tfrac{d}{ds}\varphi(W(s),M(s)) = \langle G(s),\dot W(s)\rangle +\tau^{-1}\langle P(s),\dot M(s)\rangle = -\|M(s)\|_\textnormal{nuc}.
\end{equation*}
Thus, $\varphi$ is nonincreasing along $Z(s)$. The decrease is strict if $Z(0) \notin B$. Indeed, if $\varphi(Z(T))=\varphi(Z(0))$ for some $T>0$, $\varphi$ is constant on $[0,T]$ and $M(s)=0$ for almost everywhere $s \in[0,T]$. The continuity of $M$ gives $M(s)=0$ for all $s \in [0,T]$. Thus, $\dot{M}=\tau(G-M)$ gives $G(s)=0$ for almost everywhere $s \in [0,T]$. By using the closed graph of $\DCal_f$, we obtain that $0 \in \DCal_f(W(0))$, which implies $Z(0)\in B$. Putting these pieces together yields the desired result.  

By Lemma~\ref{lem:deterministic-asymptotic-convergence-principle}, $L \subseteq B$ and $\{\varphi(Z_t)\}_{t\geq0}$ converges, which implies that every cluster point of $\{W_t\}_{t \geq 0}$ is $\DCal_f$-stationary and every cluster point of $\{M_t\}_{t \geq 0}$ is $0$. As such, $f(W_t)=\varphi(Z_t)-\tau^{-1}\|M_t\|_\textnormal{nuc}$ converges. In addition, if Assumption~\ref{assumption:convex} holds, any cluster point is globally optimal. Since $f(W_t)$ converges and $f$ is continuous, it follows that $\lim_{t\to\infty}f(W_t)=\inf_{W \in \br^{m \times n}} f(W)$. This completes the proof. 
\end{proof}

\subsection{Muon advised by gradient descent}
Muon with fixed momentum can fail on convex Lipschitz objectives, whereas Muon with an appropriately coupled momentum schedule can converge asymptotically. Whether the latter admits a finite-time complexity guarantee remains unclear.
Error feedback with momentum~\citep{Parshakova-2026-Muon} modifies Muon to obtain a finite-time rate. It tracks the discrepancy introduced by orthogonalization and feeds this error into the input of the orthogonalization at the next iteration. However, convergence of this mechanism has not been established for general nonsmooth, nonconvex objectives, and experiments suggest that it can perform worse than unmodified Muon~\citep[Section~5]{Parshakova-2026-Muon}.

We propose Muon Advised by Gradient Descent (MAGD) as an alternative modification of Muon for nonsmooth objectives. Inspired by prediction with expert advice~\citep{Cesa-2006-Prediction}, MAGD evaluates a gradient at $W_t$ and maintains two branches: $U_t$ takes the Muon update, while $S_t$ takes the gradient descent update. MAGD forms each iterate by mixing the two branches and adapts the mixing factor by comparing their decreases under the linearized loss. This rule assigns more weight to the branch with the larger predicted decrease, and hence MAGD remains close to Muon whenever the Muon branch performs better. 

\begin{algorithm}[!t]
\caption{Muon Advised by Gradient Descent (MAGD)}
\label{alg:MAGD}
\begin{algorithmic}[1]
\STATE \textbf{Input:} $W_0\in \br^{m\times n}$, $\eta_t\geq 0$, $\gamma>0$, $\rho\in(0,1)$, $\beta_t^{\rm M}\in[0,1]$, $\eta_t^{\rm M}\ge0$, $\delta_t^{\rm M}\in[0,1)$, $\sigma_t^{\rm M}\geq 0$.
\STATE \textbf{Initialization:} $U_0=S_0=W_0$, $M_0=0$, $z_0=\log(\rho/(1-\rho))$.
\FOR{$t=0,1,\ldots$}
    \STATE $\alpha_t=(1+\exp(-z_t))^{-1}$
    \STATE $W_t=(1-\alpha_t)U_t+\alpha_t S_t$
    \STATE $G_t\in \mathcal D_f(W_t)$
    \STATE $z_{t+1}=z_t-\gamma\eta_t \langle G_t,S_t-U_t\rangle$
    \STATE $S_{t+1}=S_t-\eta_t G_t$
    \STATE $M_{t+1}=\beta_t^{\rm M}M_t+(1-\beta_t^{\rm M})G_t$
    \STATE $U_{t+1}\in U_t-\eta_t^{\rm M}\operatorname{polar}_{\delta_t^{\rm M}}\left(\frac{M_{t+1}}{\|M_{t+1}\|_F+\sigma_t^{\rm M}}\right)$, with $\frac00:=0$.
\ENDFOR
\end{algorithmic}
\end{algorithm}

We next establish the asymptotic convergence of MAGD. The theorem covers two regimes. In the first regime, the GD branch asymptotically dominates the Muon branch, while the Muon branch may use a fixed or arbitrary momentum schedule. In the second regime, neither branch needs to dominate, but the momentum factor of the Muon branch approaches one as in Theorem~\ref{thm:muon-asymptotic-convergence}.

\begin{theorem}
\label{thm:magd-nonconvex-asymptotic}
Suppose that Assumption~\ref{assumption:muon-data} holds. Run Algorithm~\ref{alg:MAGD} with $\gamma>0$, and suppose that a closed ball $K\subseteq \mathbb R^{m\times n}$ contains $\{U_t\}_{t\ge0}$ and $\{S_t\}_{t\ge0}$. Set $L_K:=\sup\{\|G\|_F:W\in K,G\in\mathcal D_f(W)\}$ and $R:=\sup_{t\ge0}\|S_t-U_t\|_F$. Assume that one of the following regimes holds.
\begin{enumerate}
    \item[(i)] We have $\eta_t>0$, $\eta_t\to0$, $\sum_t\eta_t=+\infty$, and $(1-\alpha_t)\eta_t^{\rm M}/\eta_t\to0$.
    \item[(ii)] For some $\tau>0$, we have $\delta_t^{\rm M}\to0$, $\delta_t^{\rm M}\sigma_t^{\rm M}\to0$, $\beta_t^{\rm M}=1-\tau\eta_t^{\rm M}\in[0,1]$, $\eta_t^{\rm M}>0$, $\eta_t^{\rm M}\to0$, $\sum_t\eta_t^{\rm M}=+\infty$, and $\eta_t\ge0$ with $\eta_t/\eta_t^{\rm M}\to r\in[0,+\infty)$. In addition, $\gamma rL_KR\le\tau$. If $r=0$, we also assume that $\limsup_{t\to\infty}\alpha_t<1$.
\end{enumerate}
Then, every cluster point of $\{W_t\}$ lies in $\{W:0\in\mathcal D_f(W)\}$, and $\{f(W_t)\}$ converges. In regime~\textnormal{(ii)}, we additionally have $M_t\to0$.
\end{theorem}

We prove two regimes separately. Let $D_t:=S_t-U_t$, $E_t:=S_t-W_t$, and $P_{t+1}\in\operatorname{polar}_{\delta_t^{\rm M}} \left(\frac{M_{t+1}}{\|M_{t+1}\|_F+\sigma_t^{\rm M}}\right)$ such that $U_{t+1}=U_t-\eta_t^{\rm M}P_{t+1}$. Since $K$ is convex, $W_t\in K$. Local boundedness of $\mathcal D_f$ gives $L_K<+\infty$. Hence, $\sup_{t\geq 0}\|D_t\|_F\le R$, $\sup_{t\geq 0}\|G_t\|_F\le L_K$, and $\{E_t\}_{t\geq0}$ is uniformly bounded. Eq.~\eqref{eq:inexact-polar-envelope} gives $\|P_{t+1}\|_\textnormal{op}<2$, and thus $\|P_{t+1}\|_F<2\sqrt{\min\{m,n\}}$. In addition,
\begin{equation}
\label{eq:thm-asymptotic-magd-common-recursion}
\begin{aligned}
D_{t+1}-D_t&=\eta_t^{\rm M}P_{t+1}-\eta_tG_t,\\
W_{t+1}-W_t&=(\alpha_{t+1}-\alpha_t)D_t-(1-\alpha_{t+1})\eta_t^{\rm M}P_{t+1}-\alpha_{t+1}\eta_tG_t,\\
E_{t+1}-E_t&=-(\alpha_{t+1}-\alpha_t)D_t+(1-\alpha_{t+1})(\eta_t^{\rm M}P_{t+1}-\eta_tG_t).
\end{aligned}
\end{equation}
In both regimes, $\eta_t\to0$ and $z_{t+1}-z_t=-\gamma\eta_t\langle G_t,D_t\rangle=O(\eta_t)$. The sigmoid function has a globally bounded second derivative, and thus we have
\begin{equation}
\label{eq:thm-asymptotic-magd-common-alpha}
\alpha_{t+1}-\alpha_t=-\gamma\eta_t\alpha_t(1-\alpha_t)\langle G_t,D_t\rangle+O(\eta_t^2).
\end{equation}

\begin{proof}[Proof of regime (i)] We smooth the update of $\alpha_t$. Let $\Pi_R$ denote the Euclidean projection onto the Frobenius ball of radius $R$ centered at origin. For $E,G\in\br^{m\times n}$ and $\alpha\in\br$, we define
\begin{equation*}
\mathcal H_R(E,\alpha,G):=
\begin{cases}
(1-\alpha)\left\langle G,\Pi_R\left(E/(1-\alpha)\right)\right\rangle
\Pi_R\left(E/(1-\alpha)\right),&\alpha\ne1,\\
0,&\alpha=1.
\end{cases}
\end{equation*}
Since $E_t=(1-\alpha_t)D_t$ and $\|D_t\|_F\le R$, $\mathcal H_R(E_t,\alpha_t,G_t)=(1-\alpha_t)\langle G_t,D_t\rangle D_t$.
For $X=(W,E,\alpha)$, let
\begin{equation}
\label{eq:thm-asymptotic-magd-safe-branch-field}
\FCal_\textnormal{S}(X):=\bigl\{(-\alpha G-\gamma \alpha\mathcal H_R(E,\alpha,G),-(1-\alpha)G+\gamma \alpha\mathcal H_R(E,\alpha,G),-\gamma \alpha\langle G,E\rangle):G\in\mathcal D_f(W)\bigr\}.
\end{equation}
We obtain that $\FCal_\textnormal{S}$ satisfies the properties from Lemma~\ref{lem:deterministic-asymptotic-convergence-principle} since $\mathcal{H}_R$ is continuous in $(E,\alpha,G)$ and linear in $G$ and $\DCal_f$ is locally bounded, has a closed graph, and has nonempty convex values. 
Let $X_t=(W_t,E_t,\alpha_t)$ and $Y_t=\frac{X_{t+1}-X_t}{\eta_t}$. To apply Lemma~\ref{lem:deterministic-asymptotic-convergence-principle}, we have to prove that (a) $Y_t \in \FCal_\textnormal{S}^{\varepsilon_t}(X_t)$ with $\varepsilon_t \to 0$ and (b) there exists a function $\varphi_\textnormal{S}(X)$ satisfying the desired properties. 

We show (a). Let $V_t \in \FCal_\textnormal{S}(X_t)$ where $G=G_t\in\DCal_f(W_t)$. We set $\xi_t:=\frac{\alpha_{t+1}-\alpha_t}{\eta_t}+\gamma\alpha_t\langle G_t,E_t\rangle$. It follows from Eq.~\eqref{eq:thm-asymptotic-magd-common-alpha} that $\xi_t\to0$. In addition,
\begin{equation*}
\begin{aligned}
Y_t-V_t=\bigl(\xi_tD_t-(1-\alpha_{t+1})\tfrac{\eta_t^{\rm M}}{\eta_t}P_{t+1}-(\alpha_{t+1}-\alpha_t)G_t, -\xi_tD_t+(1-\alpha_{t+1})\tfrac{\eta_t^{\rm M}}{\eta_t}P_{t+1}+(\alpha_{t+1}-\alpha_t)G_t,\xi_t\bigr).
\end{aligned}
\end{equation*}
By $e^{-|z_{t+1}-z_t|}\le\tfrac{1-\alpha_{t+1}}{1-\alpha_t}=\tfrac{1+e^{z_t}}{1+e^{z_t+z_{t+1}-z_t}}\le e^{|z_{t+1}-z_t|}$ and $|z_{t+1}-z_t|=O(\eta_t)$, we have $\varepsilon_t:=\|Y_t-V_t\|_2\to0$. 

We show (b). Indeed, we define the continuous function $\varphi_\textnormal{S}(X)=f(W)$ and show that it satisfies the desired properties. Let $L_\textnormal{S}$ be the limit set of $\{X_t\}_{t\geq0}$ defined in Eq.~\eqref{eq:def-limit-set} and $B_\textnormal{S}:=\{(W,E,\alpha):0\in\mathcal D_f(W)\}$. Since $\varphi_\textnormal{S}(B_\textnormal{S}\cap L_\textnormal{S})\subseteq f(\{W:0\in\mathcal D_f(W)\})$, $\varphi_\textnormal{S}(B_\textnormal{S}\cap L_\textnormal{S})$ has empty interior. Let $X(s)=(W(s),E(s),\alpha(s))$ be a solution of $\dot X\in\FCal_\textnormal{S}(X)$ whose range lies in $L_\textnormal{S}$. For $s$ almost everywhere, the matrix $G(s):=-(\dot W(s)+\dot E(s))\in\mathcal D_f(W(s))$. The conservative field chain rule and Eq.~\eqref{eq:thm-asymptotic-magd-safe-branch-field} give
\begin{equation}
\label{eq:thm-asymptotic-magd-safe-branch-dPhi}
\tfrac{\mathrm d}{\mathrm ds}\varphi_\textnormal{S}(X(s))=-\alpha\|G\|_F^2-\gamma \alpha\langle G,\mathcal H_R(E,\alpha,G)\rangle\le0,
\end{equation}
Thus, $\varphi_\textnormal{S}$ is nonincreasing along $X(s)$. To show strict decrease outside $B_\textnormal{S}$, solving $\dot \alpha=-\gamma \alpha\langle G,E\rangle$ gives
\begin{equation}
\label{eq:ode-a}
\alpha(s)=\alpha(0)\exp\left(-\gamma\int_0^s\langle G(u),E(u)\rangle\,\mathrm du\right).
\end{equation}
Thus, $\alpha(0)>0$ implies $\alpha(s)>0$ for every finite $s\ge0$. If $\varphi_\textnormal{S}(X(T))=\varphi_\textnormal{S}(X(0))$ for some $T>0$, Eq.~\eqref{eq:thm-asymptotic-magd-safe-branch-dPhi} implies that $G(s)=0$ for almost every $s\in[0,T]$. The closed graph of $\mathcal D_f$ gives $0\in\mathcal D_f(W(0))$. It remains to consider $\alpha(0)=0$. Eq.~\eqref{eq:ode-a} gives $\alpha(s)=0$ and $W(s)=W(0)$ for all $s\ge0$, and hence $\dot E(s)=-G(s)$ for $s$ almost everywhere. Suppose that $0\notin\mathcal D_f(W(0))$. $\mathcal D_f(W(0))$ is nonempty, compact, and convex. Let $G_\star$ be the Euclidean projection of $0$ onto $\mathcal D_f(W(0))$. Then, $G_\star\ne0$ and $\langle G_\star,G\rangle\ge\|G_\star\|_F^2$ for every $G\in\mathcal D_f(W(0))$. Consequently, we have
\begin{equation*}
\langle G_\star,E(s)\rangle\le\langle G_\star,E(0)\rangle-s\|G_\star\|_F^2.
\end{equation*}
This contradicts the boundedness of $E(s)$. Hence, no solution with range in $L_\textnormal{S}$ starts in the set $\{(W,E,0)\in L_\textnormal{S}:0\notin\mathcal D_f(W)\}$. Therefore, every solution with range in $L_\textnormal{S}$ that starts outside $B_\textnormal{S}$ has $\alpha(0)>0$, and $\varphi_\textnormal{S}(X(s))<\varphi_\textnormal{S}(X(0))$ for every $s>0$. This proves the required properties of $\varphi_\textnormal{S}$.

By Lemma~\ref{lem:deterministic-asymptotic-convergence-principle}, $\{\varphi_\textnormal{S}(X_t)\}_{t\geq0}=\{f(W_t)\}_{t\geq0}$ converges. In addition, if $W_{t_j}\to\overline W$, there exists $(W_{t_{j_k}},E_{t_{j_k}},\alpha_{t_{j_k}})_{k\geq0}$ that converges to some point in $L_\textnormal{S}\subseteq B_\textnormal{S}$. Hence, we have $0\in\mathcal D_f(\overline W)$.
\end{proof}
\begin{proof}[Proof of regime (ii)]
We set $r_t:=\eta_t/\eta_t^{\rm M}$, $r:=\lim_{t\to\infty}r_t$, and $b_t:=(\alpha_{t+1}-\alpha_t)/\eta_t^{\rm M}$. For $X_t:=(W_t,D_t,M_t,\alpha_t)$, Eq.~\eqref{eq:thm-asymptotic-magd-common-recursion} gives $X_{t+1}=X_t+\eta_t^{\rm M}Y_t$, where
\begin{equation*}
Y_t:=\bigl(b_tD_t-(1-\alpha_{t+1})P_{t+1}-r_t\alpha_{t+1}G_t,P_{t+1}-r_tG_t,\tau(G_t-M_t),b_t\bigr).
\end{equation*}
Since $M_{t+1}=(1-\tau\eta_t^{\rm M})M_t+\tau\eta_t^{\rm M}G_t$ and $0\le\tau\eta_t^{\rm M}\le1$, $\{M_t\}_{t\geq0}$ and hence $\{X_t\}_{t\geq0}$ are uniformly bounded. For $X=(W,D,M,\alpha)$, we define
\begin{equation*}
\begin{split}
\FCal_{\rm M}(X):=\bigl\{(&-(1-\alpha)P-r\alpha G-\gamma r\alpha(1-\alpha)\langle G,D\rangle D,P-rG,\\
&\tau(G-M),-\gamma r\alpha(1-\alpha)\langle G,D\rangle):G\in\mathcal D_f(W),P\in\partial(\|\cdot\|_\textnormal{nuc})(M)\bigr\}.
\end{split}
\end{equation*}
We obtain that $\FCal_\textnormal{M}$ satisfies the properties of Lemma~\ref{lem:deterministic-asymptotic-convergence-principle} since $\partial(\|\cdot\|_\textnormal{nuc})$ and $\DCal_f$ are locally bounded, have closed graphs and nonempty and convex values. To apply Lemma~\ref{lem:deterministic-asymptotic-convergence-principle}, we have to prove that (a) $Y_t \in \FCal_\textnormal{M}^{\varepsilon_t}(X_t)$ with $\varepsilon_t \to 0$ and (b) there exists a function $\varphi_\textnormal{M}(X)$ satisfying the desired properties. 

We show (a). We apply the singular value truncation used in the proof of Theorem~\ref{thm:muon-asymptotic-convergence}. It gives $\widehat P_{t+1}\in\partial(\|\cdot\|_\textnormal{nuc})(\widehat M_{t+1})$ such that $\|\widehat M_{t+1}-M_{t+1}\|_\textnormal{op}\le\delta_t^{\rm M}(\|M_{t+1}\|_F+\sigma_t^{\rm M})$ and $\|\widehat P_{t+1}-P_{t+1}\|_\textnormal{op}\le\delta_t^{\rm M}$. Since $\delta_t^{\rm M}\to0$ and $\delta_t^{\rm M}\sigma_t^{\rm M}\to0$, we have $\|\widehat M_{t+1}-M_{t+1}\|_F\to0$, $\|\widehat P_{t+1}-P_{t+1}\|_F\to0$, and $\|M_{t+1}-M_t\|_F=\tau\eta_t^{\rm M}\|G_t-M_t\|_F\to0$. In addition, Eq.~\eqref{eq:thm-asymptotic-magd-common-alpha} gives $\alpha_{t+1}-\alpha_t=O(\eta_t)\to0$. We set $\widehat{X}_t:=(W_t,D_t,\widehat M_{t+1},\alpha_t)$ and let $\widehat Y_t\in\FCal_{\rm M}(\widehat X_t)$ be realized by $(G_t,\widehat P_{t+1})$. Let $e_t:=b_t+\gamma r\alpha_t(1-\alpha_t)\langle G_t,D_t\rangle$. Dividing Eq.~\eqref{eq:thm-asymptotic-magd-common-alpha} by $\eta_t^{\rm M}$, we obtain $b_t=-\gamma r_t\alpha_t(1-\alpha_t)\langle G_t,D_t\rangle+o(1)$ and hence $e_t\to0$. Denoting the four components of $Y_t$ and $\widehat{Y}_t$ by the subscripts $W,D,M,\alpha$, direct subtraction gives
\begin{align*}
Y_{t,W}-\widehat{Y}_{t,W}
={}&e_tD_t+(1-\alpha_t)(\widehat P_{t+1}-P_{t+1})+(\alpha_{t+1}-\alpha_t)P_{t+1}+\bigl((r-r_t)\alpha_t-r_t(\alpha_{t+1}-\alpha_t)\bigr)G_t,\\
Y_{t,D}-\widehat{Y}_{t,D}
={}&P_{t+1}-\widehat P_{t+1}+(r-r_t)G_t,\\
Y_{t,M}-\widehat{Y}_{t,M}
={}&\tau(\widehat M_{t+1}-M_t),\\
Y_{t,\alpha}-\widehat{Y}_{t,\alpha}
={}&e_t.
\end{align*}
Every term on the right tends to zero. In addition, since $\|\widehat{M}_{t+1}-M_{t+1}\|_F\to0$ and $\|{M}_{t+1}-M_{t}\|_F\to0$, we have $\|\widehat X_t-X_t\|\to0$. After vectorizing the matrices, we set $\varepsilon_t:=\|\widehat X_t-X_t\|_2+\|\widehat Y_t-Y_t\|_2\to0$, and thus $Y_t\in\FCal_{\rm M}^{\varepsilon_t}(X_t)$.

We show (b). If $r>0$, we set $\zeta:=0$. If $r=0$, the assumption of regime (ii) lets us choose $\zeta\in(0,1)$ such that $\alpha_t\le1-\zeta$ eventually. Let $L$ be the limit set of $\{X_t\}_{t\geq0}$ defined in Lemma~\ref{lem:deterministic-asymptotic-convergence-principle}. Since $K$ is closed, every $X\in L$ satisfies $W\in K$, $\|D\|_F\le R$, and $0\le\alpha\le1-\zeta$. We set
\begin{equation*}
B:=\{X:M=0,0\in\DCal_f(W)\},\quad \varphi_\textnormal{M}(X):=f(W)+\tfrac{1-\alpha}{\tau}\|M\|_\textnormal{nuc}+\tfrac{r\alpha}{2\tau}\|M\|_F^2.
\end{equation*}
We show that $\varphi_\textnormal{M}$ satisfies the desired properties. Since $\varphi_\textnormal{M}(B\cap L)\subseteq f(\{W:0\in\mathcal D_f(W)\})$, $\varphi_\textnormal{M}(B\cap L)$ has empty interior. In addition, $\varphi_\textnormal{M}$ is continuous. Let $X(s)=(W(s),D(s),M(s),\alpha(s))$ be a solution of $\dot X\in\FCal_{\rm M}(X)$ whose range lies in $L$. For $s$ almost everywhere, the matrices $G:=M+\tau^{-1}\dot M$ and $P:=\dot D+rG$ belong to $\mathcal D_f(W)$ and $\partial(\|\cdot\|_\textnormal{nuc})(M)$, respectively, and realize the inclusion. Then
\begin{equation*} 
\dot\alpha=-\gamma r\alpha(1-\alpha)\langle G,D\rangle,\quad \dot W=-(1-\alpha)P-r\alpha G-\gamma r\alpha(1-\alpha)\langle G,D\rangle D,\quad \dot M=\tau(G-M). \end{equation*}
Using the conservative field chain rule for $f$ and the chain rule for the nuclear norm, we obtain
\begin{align*}
\tfrac{d}{ds}\varphi_\textnormal{M}(X(s))={}&-(1-\alpha)\|M\|_\textnormal{nuc}-r\alpha\|G\|_F^2+r\alpha\langle M,G\rangle-r\alpha\|M\|_F^2\\
&-\gamma r\alpha(1-\alpha)\langle G,D\rangle^2+\tfrac{\gamma r\alpha(1-\alpha)}{\tau}\langle G,D\rangle\|M\|_\textnormal{nuc}-\tfrac{\gamma r^2\alpha(1-\alpha)}{2\tau}\langle G,D\rangle\|M\|_F^2.
\end{align*}
Letting $\theta:=\gamma rL_KR/\tau\in[0,1]$, we obtain
\begin{equation}
\label{eq:thm-asymptotic-magd-dPhi}
\begin{aligned}
\tfrac{d}{ds}\varphi_\textnormal{M}(X(s)) & \leq -(1-\alpha)^2\|M\|_\textnormal{nuc}-\tfrac{r\alpha}{2}\|G\|_F^2-\tfrac{r\alpha^2}{2}\|M\|_F^2\le0.
\end{aligned}
\end{equation}
Thus, $\varphi_\textnormal{M}$ is nonincreasing along $X(s)$.

The decrease is strict outside $B$. Indeed, if $\varphi_\textnormal{M}(X(T))=\varphi_\textnormal{M}(X(0))$ for some $T>0$, then $\varphi_\textnormal{M}$ is constant on $[0,T]$. If $r=0$, we have $1-\alpha\ge\zeta>0$, and thus the first term in Eq.~\eqref{eq:thm-asymptotic-magd-dPhi} gives $M(s)=0$ for almost every $s\in[0,T]$. We then consider the case where $r>0$. If $\alpha(s)<1$, the first term gives $M(s)=0$. If $\alpha(s)=1$, the last term gives $M(s)=0$. Hence, $M(s)=0$ for almost every $s\in[0,T]$ in either case. By continuity, $M(s)=0$ throughout $[0,T]$, and $\dot M=\tau(G-M)$ then gives $G(s)=0$ for almost every $s\in[0,T]$. The closed graph of $\mathcal D_f$ gives $0\in\mathcal D_f(W(0))$. Thus, $X(0)\in B$.

By Lemma~\ref{lem:deterministic-asymptotic-convergence-principle}, $L\subseteq B$ and $\{\varphi_\textnormal{M}(X_t)\}_{t\geq0}$ converges. Letting $\overline W$ be a cluster point of $\{W_t\}_{t\geq0}$, we choose a subsequence $W_{t_j}\to\overline W$. Boundedness of $\{(D_{t_j},M_{t_j},\alpha_{t_j})\}$ gives indices $j_k\to\infty$ such that $X_{t_{j_k}}$ converges to a point of $L\subseteq B$. Hence $0\in\mathcal D_f(\overline W)$. As such, every cluster point of the bounded sequence $\{M_t\}_{t\geq0}$ is $0$. Thus, $\varphi_\textnormal{M}(X_t)-f(W_t)=\tfrac{1-\alpha_t}{\tau}\|M_t\|_\textnormal{nuc}+\tfrac{r\alpha_t}{2\tau}\|M_t\|_F^2\to0$, completing the proof.
\end{proof}

Besides asymptotic convergence, MAGD also admits a finite-time convergence rate for convex Lipschitz objectives under Assumption~\ref{assumption:convex}. 
\begin{theorem}
\label{thm:magd-convex-finite-time-rate}
Suppose that Assumption~\ref{assumption:convex} holds and let $W^\star\in\br^{m\times n}$ satisfy $\|W_0-W^\star\|_\textnormal{op}\le R$. Assume that there exists $G> 0$ such that $|f(X)-f(Y)|\leq G\|X-Y\|_\textnormal{op}$ for any $X,Y\in\br^{m\times n}$ satisfying $\|X-W^\star\|_\textnormal{op}\leq R$ and $\|Y-W^\star\|_\textnormal{op}\leq R$. If we run Algorithm~\ref{alg:MAGD} with
\begin{equation*}
\bar\eta_t>0, \quad
\eta_t = \begin{cases}
\bar{\eta}_t / \|G_t\|_F & G_t\neq0, \\
0 & G_t=0,
\end{cases}
\quad\text{and}\quad
0\le\eta_t^{\rm M}\le\tfrac{\bar\eta_t}{2\sqrt{\min\{m,n\}}},
\end{equation*}
then, for every $t\ge1$,
\begin{equation}
\label{eq:magd-convex-finite-time-rate}
\min_{0\le s<t} f(W_s)-f(W^\star)\le \tfrac{G}{2\sum_{s=0}^{t-1}\bar\eta_s}\left[{\min\{m,n\} R^2}+\tfrac{2}{\gamma}\log(1/\rho)+\sum_{s<t}\bar\eta_s^2+\gamma\sum_{s<t}\bar\eta_s^2\left(\textstyle\sum_{q=0}^{s-1}\bar\eta_q\right)^2\right].
\end{equation}
As a consequence, for any $\epsilon>0$, if we choose $T:=\left\lceil4\min\{m,n\}(\frac{GR}{\epsilon})^2\right\rceil$, $\gamma:=3/(\min\{m,n\}R^2T)$, $\rho:=\exp(-3/(2T))$, $\bar\eta_t=R\sqrt{\min\{m,n\}/T}$, and $\eta_t^{\rm M}\in [0,R/(2\sqrt T)]$, then Algorithm~\ref{alg:MAGD} visits a time step $\hat t\in\argmin_{0\le t<T} f(W_t)$ such that $f(W_{\hat t})-f(W^\star)\le\epsilon$, and the number of subgradient evaluations is
$$O\left(\min\{m,n\}\left(\tfrac{GR}{\epsilon}\right)^2\right).$$
\end{theorem}

\begin{proof}
Let $r:=\min\{m,n\}$ and $A_s:=\sum_{q=0}^{s-1}\bar\eta_q$, with $A_0:=0$. Let $\widetilde G_s:=G_s/\|G_s\|_F$ with $0/0:=0$. Letting $P_{s+1}\in\operatorname{polar}_{\delta_s^{\rm M}}\left(\frac{M_{s+1}}{\|M_{s+1}\|_F+\sigma_s^{\rm M}}\right)$ so that $U_{s+1}=U_s-\eta_s^{\rm M}P_{s+1}$, we get $\|P_{s+1}\|_\textnormal{op}\le2$ and $\|P_{s+1}\|_F\le2\sqrt r$. Since $S_0=U_0$, $S_{s+1}=S_s-\bar\eta_s\widetilde G_s$, and $2\sqrt r\,\eta_s^{\rm M}\le\bar\eta_s$, induction gives $\|S_s-U_s\|_F\le2\sum_{q=0}^{s-1}\bar\eta_q=2A_s$.

We bound the linearized function gap. Letting $\Phi(z):=\log(1+e^{-z})$, we have $-\Phi'(z_s)=1-\alpha_s$ and $0\le\Phi''\le1/4$, and thus $\Phi(z_{s+1})
\le \Phi(z_s)+\gamma\bar\eta_s(1-\alpha_s)\langle\widetilde G_s,S_s-U_s\rangle
+\tfrac{\gamma^2\bar\eta_s^2}{8}\langle\widetilde G_s,S_s-U_s\rangle^2$.
Since $|\langle\widetilde G_s,S_s-U_s\rangle|\le2A_s$, summing from $s=0$ to $t-1$, we have
\begin{equation}
\label{eq:potential-estimate}
-\sum_{s<t}\bar\eta_s(1-\alpha_s)\langle\widetilde G_s,S_s-U_s\rangle
\le\tfrac{\Phi(z_0)-\Phi(z_t)}{\gamma}
+\tfrac{\gamma}{8}\sum_{s<t}\bar\eta_s^2\langle\widetilde G_s,S_s-U_s\rangle^2
\le\tfrac{\log(1/\rho)}{\gamma}+\tfrac{\gamma}{2}\sum_{s<t}\bar\eta_s^2A_s^2.
\end{equation}
It follows from $S_{s+1}=S_s-\bar\eta_s\widetilde G_s$ that $\bar\eta_s\langle\widetilde G_s,S_s-W^\star\rangle \le\tfrac12\left(\|S_s-W^\star\|_F^2-\|S_{s+1}-W^\star\|_F^2\right)+\tfrac{\bar\eta_s^2}{2}$. Summing over $s=0,\ldots,t-1$, we obtain
\begin{equation*}
\sum_{s<t}\bar\eta_s\langle\widetilde G_s,S_s-W^\star\rangle
\le\tfrac12\|W_0-W^\star\|_F^2-\tfrac12\|S_t-W^\star\|_F^2+\tfrac12\sum_{s<t}\bar\eta_s^2
\le\tfrac12\|W_0-W^\star\|_F^2+\tfrac12\sum_{s<t}\bar\eta_s^2.
\end{equation*}
Since $W_s=S_s-(1-\alpha_s)(S_s-U_s)$, combining the above inequality with Eq.~\eqref{eq:potential-estimate} gives
\begin{align*}
\sum_{s<t}\bar\eta_s\langle\widetilde G_s,W_s-W^\star\rangle
&=\sum_{s<t}\bar\eta_s\langle\widetilde G_s,S_s-W^\star\rangle
-\sum_{s<t}\bar\eta_s(1-\alpha_s)\langle\widetilde G_s,S_s-U_s\rangle\\
&\le\tfrac r2\|W_0-W^\star\|_\textnormal{op}^2+\tfrac{\log(1/\rho)}{\gamma}+\tfrac12\sum_{s<t}\bar\eta_s^2+\tfrac{\gamma}{2}\sum_{s<t}\bar\eta_s^2A_s^2=:C_t.
\end{align*}

We convert the inequality above to Eq.~\eqref{eq:magd-convex-finite-time-rate}. Let $j<t$ minimize $v_s:=\langle\widetilde G_s,W_s-W^\star\rangle$. Then, $v_j\le C_t/A_t$. If $C_t/A_t\ge R$, $\min_{s<t}f(W_s)-f(W^\star)\le f(W_0)-f(W^\star)\le GR\le GC_t/A_t$.
We therefore suppose that $C_t/A_t<R$. If $v_j<0$, the convexity of $f$ gives $f(W_j)-f(W^\star)<0$. If $G_j=0$, $W_j$ is a minimizer of $f$. Thus, it remains to consider $v_j\ge0$ and $G_j\ne0$. We choose $D_j\in\partial(\|\cdot\|_\textnormal{nuc})(G_j)$, $\theta_j:=\|G_j\|_Fv_j/\|G_j\|_\textnormal{nuc}$, and $Y_j:=W^\star+\theta_jD_j$ so that $\langle G_j,W_j-Y_j\rangle=0$. We then have
\begin{equation*}
0\le\theta_j\le v_j\le C_t/A_t<R,\quad \|Y_j-W^\star\|_\textnormal{op}\le R.
\end{equation*}
Convexity and Lipschitz continuity therefore proves Eq.~\eqref{eq:magd-convex-finite-time-rate} by
\begin{equation*}
f(W_j) \leq f(Y_j) \leq f(W^\star)+G\theta_j \leq f(W^\star)+GC_t/A_t.
\end{equation*}

We derive the complexity bound using Eq.~\eqref{eq:magd-convex-finite-time-rate}. For the stated choices of $T,\gamma,\rho,\bar\eta_t$, we have $A_s=sR\sqrt{r/T}$. Hence, it follows that $\sum_{s<T}\bar\eta_s^2A_s^2=R^4\frac{r^2}{T^2}\sum_{s<T}s^2\le\frac13R^4r^2T$. We also have $\sum_{s<T}\bar\eta_s=R\sqrt{rT}$, $\sum_{s<T}\bar\eta_s^2=rR^2$, and $\log(1/\rho)/\gamma=rR^2/2$. Thus, each of the four terms defining $C_T$ is at most $rR^2/2$ and $C_T\le2rR^2$. It follows from Eq.~\eqref{eq:magd-convex-finite-time-rate} that $\min_{s<T}f(W_s)-f(W^\star)\le2GR\sqrt{\tfrac rT}\le\epsilon$.
\end{proof}

\subsection{Oracle complexity in operator-norm geometry}\label{subsec:lower-bound-operator-norm}
While Theorem~\ref{thm:magd-convex-finite-time-rate} guarantees that MAGD attains an $\epsilon$-optimal function value in $O(\min\{m,n\}(GR/\epsilon)^2)$ subgradient evaluations under Assumption~\ref{assumption:convex}, is such dimension-dependent factor intrinsic to operator-norm geometry? We provide an affirmative answer by showing that this factor is unavoidable for first-order methods when $\max\{m, n\}$ is sufficiently large. Our analysis is based on the local first-order oracle model and notation of Section~\ref{subsec:oracle-model}. We recall $\mathfrak{m}_\textnormal{wcc}$ and $\mathfrak{m}_\textnormal{dc}$ (see Eq.~\eqref{eq:oracle-worst-case-complexity} and~\eqref{eq:oracle-distributional-complexity}) and define
\begin{equation*}
\XCal_\textnormal{op}(R) := \{W \in \br^{m\times n}:\|W\|_\textnormal{op} \leq R\}.
\end{equation*}
and 
\begin{equation*}
\FCal_\textnormal{op}(G,R) := \left\{f:\br^{m\times n} \to \br \ \middle| \
\begin{array}{l}
f\textnormal{ is convex and }G \textnormal{-Lipschitz with respect to}\\
\|\cdot\|_\textnormal{op} \textnormal{ on an open set containing }\XCal_\textnormal{op}(R)
\end{array}
\right\}.
\end{equation*}
\begin{theorem}\label{thm:oracle-complexity}
If $\max\{m,n\}>(GR/\epsilon)^2>1$, we have
\begin{equation*}
\Omega\left(\min\{m,n\}\left(\tfrac{GR}{\epsilon}\right)^2\right) = \mathfrak{m}_\textnormal{dc}(\XCal_\textnormal{op}(R),\FCal_\textnormal{op}(G,R),\mathsf{O},\epsilon) \leq \mathfrak{m}_\textnormal{wcc}(\XCal_\textnormal{op}(R),\FCal_\textnormal{op}(G,R),\mathsf{O},\epsilon) = O\left(\min\{m,n\}\left(\tfrac{GR}{\epsilon}\right)^2\right), 
\end{equation*}
where $\mathsf{O}$ is a local first-order oracle on $\FCal_\textnormal{op}(G, R)$.
\end{theorem}
We divide the proof into two steps: (i) construct the hard instances in operator-norm geometry and reformulate the resulting optimization problem as the string guessing problem \citep[Oracle~III.2 and Proposition~III.3]{Braun-2017-Lower}, and (ii) apply the lower bound for solving the string guessing problem to yield the desired result. 
\begin{definition}[string guessing oracle] \label{def:string-guessing-oracle}
Fixing $M \geq 1$ and any hidden sign vector $S \in \{\pm1\}^M$, a query to the string guessing oracle on $S$ is $(a_i,g_i)_{i=1}^\ell$, where $0 \leq \ell \leq M$, indices $\{a_i\}_{i=1}^\ell \subseteq \{1,2,\ldots,M\}$, and guesses $\{g_i\}_{i=1}^\ell\in\{\pm1\}^\ell$. If there exists $1 \leq i \leq \ell$ such that $S_{a_i}\ne g_i$, the oracle returns the first mismatch $\inf\{1 \leq i \leq \ell: S_{a_i} \neq g_i\}$. Otherwise, the oracle returns \textnormal{\textsc{equal}}.
\end{definition}
For simplicity, we define $[N]:=\{1,\ldots,N\}$ for any integer $N \geq 1$.
A string guessing algorithm recovers $S \in \{\pm1\}^M$ by querying the string guessing oracle $\OCal$. For $q \geq 1$, $t \in [q]$, $0 \leq \ell_t \leq M$, $\{a_i^t\}_{i=1}^{\ell_t} \subseteq [M]$, $\{g_i^t\}_{i=1}^{\ell_t}\in\{\pm1\}^{\ell_t}$ and $Q_t=(a_i^t,g_i^t)_{i=1}^{\ell_t}$, it keeps $H_q = (Q_t,\OCal(Q_t))_{t\in[q]}$ with $H_0=\emptyset$, and uses a decision rule based on $H_q$ that either returns some $\widehat{S} \in \{\pm1\}^M$ or the next query $(a_i^{q+1},g_i^{q+1})_{i=1}^{\ell_{q+1}}$. A consequence of~\cite[Proposition~III.3]{Braun-2017-Lower} shows that $\OCal$ hides useful information from any randomized algorithm attempting to recover $S \sim \textnormal{Unif}(\{\pm1\}^M)$. We define the binary entropy by
\begin{equation*}
\HH: [0,1] \to \br,\quad \HH(p):=-p\log_2p-(1-p)\log_2(1-p). 
\end{equation*}
We summarize this result in the following lemma. 
\begin{lemma} \label{lem:string-guessing}
Let $M\geq 1$ and $S \sim \textnormal{Unif}(\{\pm1\}^M)$. For any string guessing algorithm $\mathsf{A}$ that accesses $S$ only through the string guessing oracle and stops with an output $\widehat{S}(\mathsf{A},S)$ after $\mathcal{T}_{\mathsf{A}}(S)$ queries, we have
\begin{equation*}
\EE[\TCal_{\mathsf{A}}(S)]\geq \tfrac{1}{2}((1-\bp\{\widehat{S}(\mathsf{A},S)\neq S\})M-\HH(\bp\{\widehat{S}(\mathsf{A},S)\neq S\})).
\end{equation*}
As a consequence, the lower bound on $\EE_S[\TCal_{\mathsf{A}}(S)]$ increases in $M$ for given $\bp\{\widehat{S}(\mathsf{A},S)\neq S\} \leq \frac{1}{2}$ fixed, and decreases in $\bp\{\widehat{S}(\mathsf{A},S)\neq S\} \leq \frac{1}{2}$ given $M$ fixed. 
\end{lemma}
The following proposition shows that optimizing a max-of-affine function whose $M$ affine pieces encode a hidden vector $S \sim \operatorname{Unif}(\{\pm1\}^M)$ is at least as hard as recovering $S$ through the string guessing oracle. We recall from Section~\ref{subsec:oracle-model} that $\mathcal{A}(\mathsf O)$ is the set of deterministic algorithms using oracle $\mathsf{O}$, and $\mathsf{T}_{\mathsf A}(f,\epsilon)$ is the number of queries made by $\mathsf A\in\mathcal{A}(\mathsf{O})$ to reach an $\epsilon$-optimal solution of $f$.

\begin{proposition} \label{prop:lb-affine-reduction}
Let $M \geq 1$, $\tau>0$, and $A=(A_a)_{a \in [M]}\in(\br^{m\times n})^M$ such that $\|A_a\|_\textnormal{nuc} \leq 1$ for every $a \in [M]$. For $s \in \{\pm1\}^M$ and $\theta \in [0,\tau]^M$, let $F_{s,\theta}: \br^{m\times n} \to \br$, $F_{s,\theta}(X):=\max_{a \in [M]}[s_a\langle A_a,X\rangle+\theta_a]$. We draw $S \sim \operatorname{Unif}(\{\pm1\}^M)$ and $\Theta \sim \operatorname{Unif}([0,\tau]^M)$ independently. For any $\delta_0>0$, any local first-order oracle $\mathsf{O}$ on $\FCal_{\textnormal{op}}(1,1)$, and any $\mathsf{A} \in \ACal(\mathsf{O})$ such that $\bp(\mathsf{T}_{\mathsf{A}}(F_{S,\Theta},\delta_0)<+\infty)=1$, for almost every $\theta \in [0,\tau]^M$, there exists a string guessing algorithm $\mathsf{B}_\theta$ such that the string guessing oracle responses that $\mathsf{B}_\theta$ receives on any $s\in\{\pm1\}^M$ determine the queries, oracle responses, and output of $\mathsf{A}$ on $F_{s,\theta}$ with oracle $\mathsf{O}$, $\TCal_{\mathsf{B}_\theta}(s) \leq \mathsf{T}_{\mathsf{A}}(F_{s,\theta},\delta_0)$, and $\bp\{\widehat{S}(\mathsf{B}_\theta,S)\neq S\}\leq 1-|\GCal_{A,4\tau+\delta_0}|/2^M$, where
\begin{equation*}
\GCal_{A,\gamma} := \{s\in\{\pm1\}^M: \textnormal{ there exists } X \in \XCal_{\textnormal{op}}(1)\textnormal{ such that } s_a\langle A_a,X\rangle \leq -\gamma\textnormal{ for all } a \in [M]\}.
\end{equation*}
\end{proposition}

\begin{proof}
For $a \in [M]$, $s \in \{\pm1\}^M$, and $\theta \in [0,\tau]^M$, we define the function $F_{s,\theta}(X)=\max_{a\in[M]}\phi_a^{s,\theta}(X)$ where $\phi_a^{s,\theta}(X):=s_a\langle A_a,X\rangle+\theta_a$. Accordingly, we define $I_{s,\theta}(X):= \{a \in [M]: \phi_a^{s,\theta}(X)=F_{s,\theta}(X)\}$. Given any algorithm $\mathsf{A} \in \ACal(\mathsf{O})$, we define $I^t_{s,\theta}$ as the index of the affine function for its $t^\textnormal{th}$ query $X^t_{s,\theta}$, i.e., $I^t_{s,\theta}:=I_{s,\theta}(X^t_{s,\theta})$. If $\mathsf{A}$ stops before $t^\textnormal{th}$ iteration, $I^t_{s,\theta}:=\emptyset$. We also define $J^t_{s,\theta}=\bigcup_{q=1}^t I^q_{s,\theta}$ with $J^0=\emptyset$. 

In what follows, we define $\ECal:=\{\theta \in [0,\tau]^M: |I^t_{s,\theta}\setminus J^{t-1}_{s,\theta}| \leq 1, \mathsf{T}_{\mathsf{A}}(F_{s,\theta},\delta_0)<+\infty, \forall s\in\{\pm1\}^M, \forall t \geq 1\}$ and show that (i) for $\forall \theta \in \ECal$, there exists a string guessing algorithm $\mathsf{B}_\theta$ with an algorithm $\mathsf{A}'$ such that the string guessing oracle responses that $\mathsf{B}_\theta$ receives produces the oracle responses that $\mathsf{A}'$ receives. $\mathsf{A}'$ and $\mathsf{A}$ are based on different oracles but have the same oracle responses if we run $\mathsf{A}$ on $F_{s,\theta}$ with oracle $\mathsf{O}$ for every $s \in \{\pm1\}^M$; (ii) $\TCal_{\mathsf{B}_\theta}(s) \leq \mathsf{T}_{\mathsf{A}}(F_{s,\theta},\delta_0)$ and $\bp\{\widehat{S}(\mathsf{B}_\theta,S) \neq S\}\leq 1-|\GCal_{A,4\tau+\delta_0}|/2^M$ for every $s \in \{\pm1\}^M$; and (iii) $\bp(\ECal)=1$. For $z \in \br$, we define $\operatorname{sgn}_0(z)=1$ if $z \geq 0$ and $\operatorname{sgn}_0(z)=-1$ otherwise. 

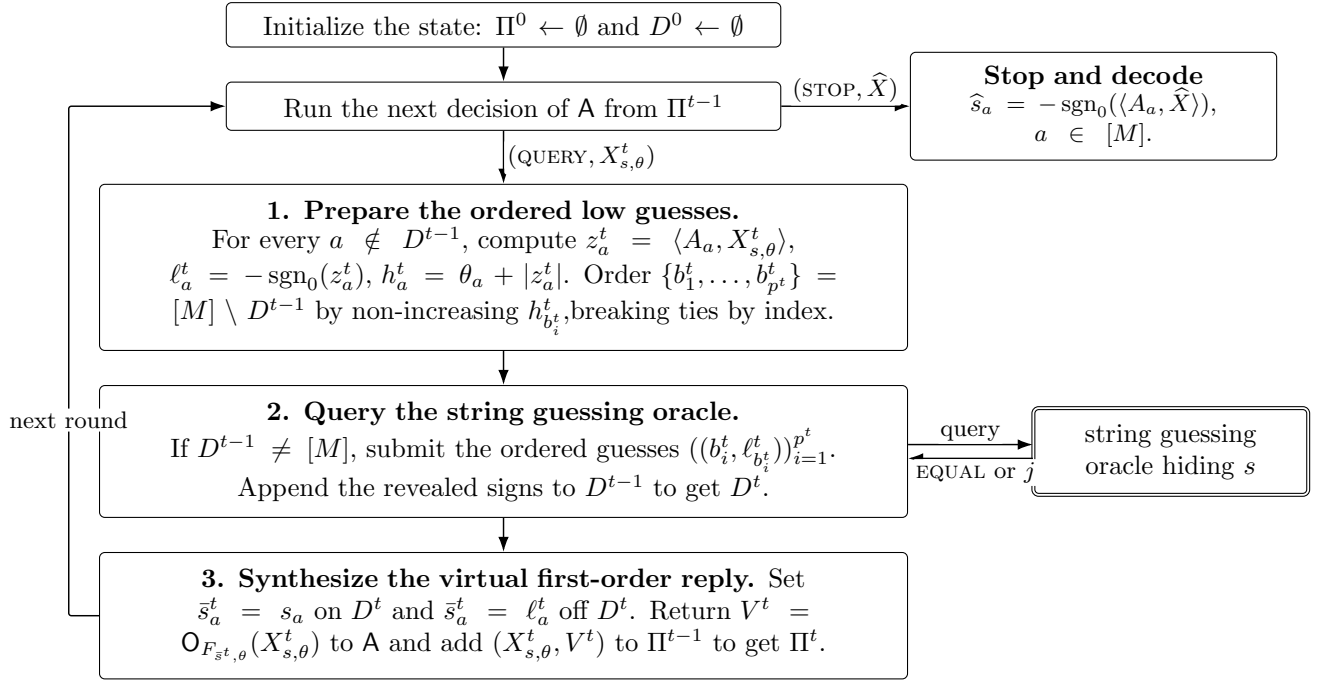
\begin{figure}[!t]
\centering
\begin{tikzpicture}[
    font=\small, >={Latex[length=2mm,width=1.4mm]}, 
    flow/.style={
        ->,
        line width=.55pt, 
        rounded corners=1pt
    },
    init/.style={
        draw,
        rounded corners=2pt,
        align=center,
        inner xsep=6pt,
        inner ysep=4pt,
        text width=.40\linewidth
    },
    decision/.style={
        draw,
        rounded corners=2pt,
        align=center,
        inner xsep=6pt,
        inner ysep=5pt,
        text width=.40\linewidth
    },
    process/.style={
        draw,
        rounded corners=2pt,
        align=center,
        inner xsep=7pt,
        inner ysep=6pt,
        text width=.59\linewidth
    },
    stopbox/.style={
        draw,
        rounded corners=2pt,
        align=center,
        inner xsep=5pt,
        inner ysep=5pt,
        text width=.26\linewidth
    },
    oracle/.style={
        draw,
        double,
        double distance=.7pt,
        rounded corners=2pt,
        align=center,
        inner xsep=5pt,
        inner ysep=6pt,
        text width=.19\linewidth
    },
    edge label/.style={
        fill=white,
        inner sep=1.2pt,
        font=\footnotesize
    }
]

\node[init] (init)
    {Initialize the state:
    $\Pi^0\gets\emptyset$ and $D^0\gets\emptyset$};

\node[decision, below=4.5mm of init] (decide)
    {Run the next decision of $\mathsf{A}$
    from $\Pi^{t-1}$};

\node[stopbox, right=17mm of decide] (stop)
    {\textbf{Stop and decode}\\[-.2ex]
    {\footnotesize
        $\widehat{s}_a = -\operatorname{sgn}_0(\langle A_a,\widehat{X}\rangle)$,\\[-.2ex]$a \in [M]$.
    }
    };

\node[process, below=7mm of decide] (prepare)
    {\textbf{1. Prepare the ordered low guesses.}\\[-.2ex]
    For every $a \notin D^{t-1}$, compute 
    $z_a^t=\langle A_a,X^t_{s,\theta}\rangle$, 
    $\ell_a^t=-\operatorname{sgn}_0(z_a^t)$, 
    $h_a^t=\theta_a+|z_a^t|$.
    Order $\{b_1^t,\ldots,b_{p^t}^t\} = [M]\setminus D^{t-1}$ by non-increasing $h_{b_i^t}^t$,breaking ties by index.
    };

\node[process, below=4.5mm of prepare] (reveal)
    {\textbf{2. Query the string guessing oracle.}\\
    If $D^{t-1}\neq[M]$, submit the ordered guesses
    $((b_i^t,\ell_{b_i^t}^t))_{i=1}^{p^t}$.
    Append the revealed signs to $D^{t-1}$ to get $D^{t}$.
    };

\node[oracle, right=16.5mm of reveal] (sgo)
    {string guessing\\oracle hiding $s$};

\node[process, below=4.5mm of reveal] (reply)
    {\textbf{3. Synthesize the virtual first-order reply.}
    Set $\bar{s}_a^t=s_a$ on $D^{t}$ and
    $\bar{s}_a^t=\ell_a^t$ off $D^{t}$.
    Return $V^t=\mathsf{O}_{F_{\bar{s}^t,\theta}}(X^t_{s,\theta})$ to $\mathsf{A}$ and
    add $(X^t_{s,\theta},V^t)$ to $\Pi^{t-1}$ to get $\Pi^t$.
    };

\draw[flow] (init) -- (decide);

\draw[flow]
    (decide.south) --
    node[edge label, right]
    {$(\textnormal{\textsc{query}},X^t_{s,\theta})$}(prepare.north);

\draw[flow]
    (decide.east) --
    node[edge label, above]
    {$(\textnormal{\textsc{stop}},\widehat{X})$}(stop.west);

\draw[flow] (prepare) -- (reveal);
\draw[flow] (reveal) -- (reply);

\draw[flow]
    ([yshift=2.6pt]reveal.east) --
    node[edge label, above] {query}([yshift=2.6pt]sgo.west);

\draw[flow]
    ([yshift=-2.6pt]sgo.west) --
    node[edge label, below, xshift=.7mm] {\textsc{equal} or $j$}([yshift=-2.6pt]reveal.east);

\draw[flow]
    (reply.west) -- ++(-4mm,0) |-
    node[edge label, pos=.18, above] {next round}(decide.west);
\end{tikzpicture}

\caption{Reduction of a convex nonsmooth optimization algorithm $\mathsf{A}$ to a string-guessing algorithm $\mathsf{B}$}
\label{fig:string-guessing}
\end{figure}

We prove (i) by characterizing $\mathsf{B}_\theta$ and $\mathsf{A}'$ as shown in Figure~\ref{fig:string-guessing} and showing the desired properties. Indeed, we denote by $D^{t-1} \subseteq  [M]$ the indices of the signs that are known to $\mathsf{B}_\theta$ before $\mathsf{A}'$ makes its $t^\textnormal{th}$ query with $D^0=\emptyset$. For $a \notin D^{t-1}$ and a query $X^t_{s,\theta}$, we define $z_a^t:=\langle A_a,X^t_{s,\theta}\rangle$, the sign of the lower sign and higher value $\ell_a^t=-\operatorname{sgn}_0(z_a^t)$ and $h_a^t=\theta_a+|z_a^t|$. Suppose that the $t^\textnormal{th}$ decision is to query $X^t_{s,\theta}$. Then, we order the unknown indices as $(b_1^t, \ldots, b_{M-|D^{t-1}|}^t)$ according to non-increasing values $h_{b_i^t}^t$. If $|D^{t-1}|=M$, we set $D^{t}=D^{t-1}$. If $|D^{t-1}|<M$, $\mathsf{B}_\theta$ queries the string guessing oracle with $((b_i^t,\ell_{b_i^t}^t))_{i=1}^{M-|D^{t-1}|}$. A response \textnormal{\textsc{equal}} gives $s_{b_i^t}=\ell_{b_i^t}^t$ for every $i \in [M-|D^{t-1}|]$ and $D^{t}=[M]$. A first mismatch at position $j_t$ gives $s_{b_i^t}=\ell_{b_i^t}^t$ for $i<j_t$ and $s_{b_{j_t}^t}=-\ell_{b_{j_t}^t}^t$. Finally, we update $D^{t}:=D^{t-1}\cup\{b_1^t,\ldots,b_{j_t}^t\}$ and define $V^t:=\mathsf{O}_{F_{\bar{s}^t,\theta}}(X^t_{s,\theta})$ where
\begin{equation*}
\bar{s}_a^t:= \begin{cases}
s_a, & \textnormal{if } a\in D^t,\\ \ell_a^t, & \textnormal{otherwise},
\end{cases}
\end{equation*}
We feed $V^t$ to $\mathsf{A}'$ and $(X^t_{s,\theta},V^t)$ to an ordered list $\Pi^{t-1}$ to get $\Pi^t$ that stores previous queries and oracle responses with $\Pi^{0}=\emptyset$. If the $t^\textnormal{th}$ decision is to output $\widehat{X}$, we output $\widehat{s}_a:=-\operatorname{sgn}_0(\langle A_a,\widehat{X}\rangle)$ for $a \in [M]$. Since $\mathsf{A}'$ receives $V^t$, we proceed to the $(t+1)^\textnormal{th}$ decision of $\mathsf{A}'$ to query $X^{t+1}$ or stop with some $\widehat{X}$. 

By definition, $\mathsf{A}'$ and $\mathsf{A}$ are based on different oracles. Then, we show that they have the same oracle responses if we run $\mathsf{A}$ on $F_{s,\theta}$ with oracle $\mathsf{O}$ for every $s \in \{\pm1\}^M$. It suffices to show that $J^t_{s,\theta} \subseteq D^t$ and $\Pi^t$ coincides with the queries and oracle responses that running $\mathsf{A}$ on $F_{s,\theta}$ with oracle $\mathsf{O}$ would receive. We first prove $J^t_{s,\theta} \subseteq D^t$ by induction. Indeed, $J^0 \subseteq D^0$ since $J^0=D^0=\emptyset$. If $D^{t-1}=[M]$ or the oracle replies \textnormal{\textsc{equal}}, we have $J^t_{s,\theta} \subseteq D^t=[M]$. Thus, we consider the case when the oracle replies with a mismatch $b:=b_{j_t}^t$. By definition, $s_b=-\ell_{b}^t$ and $\phi_b^{s,\theta}(X^t_{s,\theta})=h_b^t$. We claim that $I^t_{s,\theta} \subseteq D^t$. Indeed, for any $a \notin D^t$, we have $h_a^t \leq h_b^t$. By definition, we have $\phi_a^{s,\theta}(X^t_{s,\theta}) \leq h_a^t$. If $\phi_b^{s,\theta}(X^t_{s,\theta})<F_{s,\theta}(X^t_{s,\theta})$, we have $h_b^t<F_{s,\theta}(X^t_{s,\theta})$. Thus, $\phi_a^{s,\theta}(X^t_{s,\theta}) \leq h_a^t \leq h_b^t<F_{s,\theta}(X^t_{s,\theta})$ which implies $a \notin I^t_{s,\theta}$. If $\phi_b^{s,\theta}(X^t_{s,\theta})=F_{s,\theta}(X^t_{s,\theta})$, we have $b\in I^t_{s,\theta}\setminus J^{t-1}_{s,\theta}$ since $J^{t-1}_{s,\theta} \subseteq D^{t-1}$. Since $\theta \in \ECal$, we have $|I^t_{s,\theta} \setminus J^{t-1}_{s,\theta}|\leq 1$ which implies $a\notin I^t_{s,\theta}$. Since $J^t_{s,\theta}=J^{t-1}_{s,\theta}\cup I^t_{s,\theta}$, we have $J^t_{s,\theta} \subseteq D^t$. We then prove that $\Pi^t$ coincides with the queries and oracle responses that running $\mathsf{A}$ on $F_{s,\theta}$ with oracle $\mathsf{O}$ would receive by induction. Indeed, we have
\begin{equation*}
\phi_a^{\bar{s}^t,\theta}(X^t_{s,\theta}) = \theta_a-|z_a^t| \leq \phi_a^{s,\theta}(X^t_{s,\theta})<F_{s,\theta}(X^t_{s,\theta}), \textnormal{ for any } a \notin D^t.
\end{equation*}
which implies that $F_{s,\theta}=F_{\bar{s}^t,\theta}$ on the neighborhood of $X^t_{s,\theta}$ and $\mathsf{O}_{F_{s,\theta}}(X^t_{s,\theta})=\mathsf{O}_{F_{\bar{s}^t,\theta}}(X^t_{s,\theta})=V^t$.

We prove (ii). By definition, $\mathsf{B}_\theta$ makes at most one query per query of $\mathsf{A}$. Combining it with $\theta \in \ECal$ yields that $\TCal_{\mathsf{B}_\theta}(s) \leq \mathsf{T}_{\mathsf{A}}(F_{s,\theta},\delta_0) < +\infty$ for all $s \in \{\pm1\}^M$. We fix $s \in \GCal_{A,4\tau+\delta_0}$ and choose $X_s \in \XCal_\textnormal{op}(1)$ from the definition of $\GCal_{A,4\tau+\delta_0}$. Since $\theta_a \leq \tau$, we have $F_{s,\theta}(X_s) \leq -3\tau-\delta_0$. If $\widehat{X}$ is the $\delta_0$-optimal point output by $\mathsf{A}$, we have $F_{s,\theta}(\widehat{X}) \leq -3\tau < 0$. Since $\theta_a \geq 0$, this implies $s_a\langle A_a, \widehat{X}\rangle < 0$ for every $a \in [M]$, and $\widehat{s}_a=-\operatorname{sgn}_0(\langle A_a,\widehat{X}\rangle) = s_a$. Therefore, $\mathsf{B}_\theta$ recovers every $s \in \GCal_{A,4\tau+\delta_0}$, and we have $\bp\{\widehat{S}(\mathsf{B}_\theta,S) \neq S\} \leq 1-|\GCal_{A,4\tau+\delta_0}|/2^M$. 

We prove (iii). Since $\bp(\mathsf{T}_{\mathsf{A}}(F_{S,\Theta},\delta_0)<+\infty)=1$ and $S \in \{\pm1\}^M$ is a finite-valued random variable, $\mathsf{T}_{\mathsf{A}}(F_{s,\theta},\delta_0)<+\infty$ for any $s \in \{\pm1\}^M$ and almost every $\theta$. It suffices to show that $\bp\{|I^t_{s,\Theta}\setminus J^{t-1}_{s,\Theta}|\leq 1 \textnormal{ for all } t \geq 1\}=1$ for any $s\in\{\pm1\}^M$. 

Fixing $t \geq 1$, $s \in \{\pm1\}^M$ and $a \neq b$, we define $\Theta_{-a}:=(\Theta_1,\ldots,\Theta_{a-1},\Theta_{a+1},\ldots,\Theta_M)$ and claim that $\bp(a,b\in (I^t_{s,\Theta} \setminus J^{t-1}_{s,\Theta}) \mid \Theta_{-a})=0$. For $u \in [0,\tau]$, we let $\theta(u)\in[0,\tau]^M$ satisfy $\theta_a(u)=u$ and $\theta_{-a}(u)=\Theta_{-a}$, $X^q(u):=X^q_{s,\theta(u)}$ and define $E=\{u \in [0,\tau]: a,b\in I^t_{s,\theta(u)} \setminus J^{t-1}_{s,\theta(u)}\}$. In what follows, we show that $E$ is countable. Indeed, by induction over $t$, we show that $X^t(u)$ is locally constant for $u\in E$. Indeed, suppose that $X^q(u')=X^q(u)$ for every $u'$ in a neighborhood of $u$ for some $q<t$. Since $a \notin J^{t-1}_{s,\theta(u)}$, we have $\phi^{s,\theta(u)}_a(X^q(u))<F_{s,\theta(u)}(X^q(u))$ for every $q<t$. This together with the continuity implies that, after shrinking this neighborhood if necessary, $\phi^{s,\theta(u')}_a(X)<\max_{c\neq a}\phi^{s,\theta(u')}_c(X)$ for every $u'$ near $u$ and every $X$ in a neighborhood of $X^q(u)$. For every $c \neq a$, we have $\phi^{s,\theta(u')}_c=\phi^{s,\theta(u)}_c$ since $\theta_c(u')=\theta_c(u)=\Theta_c$. This implies $F_{s,\theta(u')}=F_{s,\theta(u)}$ on the neighborhood of $X^q(u)$. By locality of $\mathsf{O}$, the oracle response at $X^q(u')=X^q(u)$ remains unchanged, and thus the $(q+1)^\textnormal{th}$ query also remains unchanged. Since the first query is independent of $u$, this induction shows that $X^t(\cdot)$ is constant on a neighborhood of $u$. This yields the desired result and further implies that $\Theta_b+s_b\langle A_b,X^t(\cdot)\rangle-s_a\langle A_a,X^t(\cdot)\rangle$ is constant on this neighborhood. If $u\in E$, $a,b\in I^t_{s,\theta(u)}\setminus J^{t-1}_{s,\theta(u)}$ implies $u=\Theta_b+s_b\langle A_b,X^t(u)\rangle-s_a\langle A_a,X^t(u)\rangle$. Thus, at most one point of $E$ satisfies this equality in this neighborhood. Putting these pieces together yields that $E$ is countable. Since $\Theta_a$ is conditionally uniform on $[0,\tau]$, we obtain $\bp(\Theta_a\in E\mid\Theta_{-a})=0$, which proves $\bp(a,b\in I^t_{s,\Theta}\setminus J^{t-1}_{s,\Theta}\mid\Theta_{-a})=0$.

Finally, we have $\bp(a,b\in (I^t_{s,\Theta}\setminus J^{t-1}_{s,\Theta}))=0$ for any $s\in\{\pm1\}^M$. This implies that $\bp\{|I^t_{s,\Theta}\setminus J^{t-1}_{s,\Theta}|\leq 1 \textnormal{ for all } t \geq 1\}=1$ for any $s\in\{\pm1\}^M$ and completes the proof.  
\end{proof}
 
\begin{proof}[Proof of Theorem~\ref{thm:oracle-complexity}]
Eq.~\eqref{eq:oracle-complexity-order} and Theorem~\ref{thm:magd-convex-finite-time-rate} guarantee $\mathfrak{m}_\textnormal{dc} \leq \mathfrak{m}_\textnormal{wcc}=O(\min\{m,n\}(\frac{GR}{\epsilon})^2)$. Thus, it suffices to prove $\mathfrak{m}_\textnormal{dc}(\XCal_\textnormal{op}(R),\FCal_\textnormal{op}(G,R),\mathsf{O},\epsilon)=\Omega(\min\{m,n\}(\tfrac{GR}{\epsilon})^2)$. 

To apply Proposition~\ref{prop:lb-affine-reduction}, we claim that, for any $\delta \in (\frac{1}{\sqrt{\max\{m,n\}}},1)$, there exists a universal constant $c>0$ such that there exist $\tau>0$, $M \geq \max\{1,c\min\{m,n\}\delta^{-2}\}$, and $A=(A_a)_{a \in [M]}  \in (\br^{m\times n})^M$ with $\|A_a\|_\textnormal{nuc} \leq 1$ and $1-|\GCal_{A,4\tau+\delta}|/2^M \leq \frac{1}{8}$. Indeed, the claim implies 
\begin{equation}
\label{eq:GR11}
\mathfrak{m}_\textnormal{dc}(\XCal_\textnormal{op}(1),\FCal_\textnormal{op}(1,1),{\mathsf{O}},\delta)= \Omega(\min\{m,n\}\delta^{-2})
\end{equation}
for any local oracle ${\mathsf{O}}$ on $\FCal_\textnormal{op}(1,1)$ and any $\delta\in(\frac{1}{\sqrt{\max\{m,n\}}},1)$ using the hard function distribution $F_{S,\Theta}$ over $\FCal_\textnormal{op}(1,1)$ defined in Proposition~\ref{prop:lb-affine-reduction}. For any $\mathsf{A} \in \ACal({\mathsf{O}})$, if $\mathsf{T}_{\mathsf{A}}(F_{S,\Theta},\delta)=+\infty$ with any positive probability, $\EE[\mathsf{T}_{\mathsf{A}}(F_{S,\Theta},\delta)]=+\infty$ and the lower bound holds. Otherwise, Proposition~\ref{prop:lb-affine-reduction} gives, for almost every $\theta$, a string guessing algorithm $\mathsf{B}_\theta$ reproducing $\mathsf{A}$'s queries and responses and satisfying $\mathcal T_{\mathsf{B}_\theta}(s) \leq \mathsf{T}_{\mathsf{A}}(F_{s,\theta},\delta)$ for all $s \in \{\pm 1\}^M$ and $\bp\{\widehat{S}(\mathsf{B}_\theta,S)\neq S\} \leq \frac{1}{8}$, and Lemma~\ref{lem:string-guessing} gives $\EE_S[\TCal_{\mathsf{B}_\theta}(S)] \geq \frac{1}{2}(\frac{7}{8}M-\HH(\frac{1}{8}))$. Taking expectation over $\Theta$ yields $\EE[\mathsf{T}_{\mathsf{A}}(F_{S,\Theta},\delta)] \geq (\frac{7}{16}-\frac{1}{2}\HH(\tfrac{1}{8}))c\min\{m,n\}\delta^{-2}$. This holds for any $\mathsf A\in\mathcal A({\mathsf O})$, proving Eq.~\eqref{eq:GR11}.

In what follow, we prove the desired claim. Indeed, without loss of generality, we assume $m\leq n$. We derive a key probability inequality to separate low and high accuracy regimes. Let $e_i^{(d)}$ denote the $i^{\rm th}$ standard basis vector in $\br^d$. 
The uniform distribution on the $\br^n$ sphere with radius $\sqrt n$ is isotropic and has a uniformly bounded sub-gaussian norm \citep[Theorem~3.4.6]{Vershynin-2018-High}.
Hence, \citep[Theorem~4.6.1]{Vershynin-2018-High} gives an absolute constant $C>0$ such that every random $m\times n$ matrix $B$ with independent rows uniform on the unit sphere satisfies
\begin{equation}\label{eq:lb-spherical-row-bound}
\bp\{\|B\|_\textnormal{op}\le\sqrt{m/n}+2C\} \geq 1-2e^{-n}.
\end{equation}
We set $c=(8c_0^2)^{-1}$ where $c_0=1+2C$ and consider two cases. First, we assume $\delta \geq 1/(2c_0)$. By setting $M:=m$, $A_i:=e_i^{(m)}(e_i^{(n)})^\top$ and $X_s:=-\sum_{i=1}^m s_iA_i$ for $\forall s\in\{\pm1\}^m$, we have $\|A_i\|_\textnormal{nuc}=1$, $\|X_s\|_\textnormal{op}=1$ and $s_i\langle A_i,X_s\rangle=-1$. By setting $\tau:=(1-\delta)/4$, we have $\GCal_{A,4\tau+\delta}=\{\pm1\}^M$ and $M \geq cm\delta^{-2}$. Second, we assume $1/\sqrt{n}<\delta<1/(2c_0)$. By setting $k=\lfloor(2c_0\delta)^{-2}\rfloor$, we have $k<n$, $n>4c_0^2 \geq 4$ and $1-2e^{-n} \geq \frac{7}{8}$. The existence of $A_i$ can be established by the following argument. Indeed, letting $Q_1,\ldots,Q_m\in \br^{n\times n}$ be independent orthogonal matrices distributed uniformly with respect to the Haar measure, we set $u_{ij}:=Q_ie_j^{(n)}$, $A_{ij}:=e_i^{(m)}u_{ij}^\top$, and $B_s:=\tfrac{1}{\sqrt{k}}\sum_{i=1}^m\sum_{j=1}^k s_{ij}A_{ij}$. For any given $s \in \{\pm1\}^{m\times k}$, the $i^\textnormal{th}$ row of $B_s$ is $(Q_iw_i)^\top$, where $w_i:=k^{-1/2}\sum_{j=1}^k s_{ij}e_j^{(n)}$ is a unit vector. Thus, the rows of $B_s$ are independent and uniform on the unit sphere. Since $\sqrt{m/n}+2C \leq c_0$, Eq.~\eqref{eq:lb-spherical-row-bound} guarantees $\bp_Q\{\|B_s\|_\textnormal{op}\le c_0\} \geq \frac{7}{8}$ which implies $\EE_Q[2^{-mk}|\{s \in \{\pm1\}^{m\times k}: \|B_s\|_\textnormal{op} \leq c_0\}|] \geq \frac{7}{8}$. Therefore, some deterministic realization of $Q_1,\ldots,Q_m$ makes the quantity inside the expectation at least $\frac{7}{8}$. We fix such a deterministic $Q_1,\dots,Q_m$ and let $M=mk$. Then, we define $X_s:=-\frac{B_s}{c_0}$ where
\begin{equation*}
\GCal:=\{s\in\{\pm1\}^{m\times k}:\|B_s\|_\textnormal{op} \leq c_0\}.
\end{equation*}
Thus, $|\GCal|/2^M \geq \frac{7}{8}$ and $\|X_s\|_\textnormal{op} \leq 1$ for all $s \in \GCal$. In addition, we have $\|A_{ij}\|_\textnormal{nuc}=1$ and $\langle A_{ij},A_{\ell q}\rangle =\langle e_i^{(m)},e_\ell^{(m)}\rangle\langle u_{ij},u_{\ell q}\rangle =\mathbf{1}_{\{i=\ell,j=q\}}$. As a consequence, $s_{ij}\langle A_{ij},X_s\rangle = -\frac{1}{c_0\sqrt{k}}$. Letting $\tau:=\frac{1}{4c_0\sqrt k}-\frac{1}{4}\delta$, we obtain that any $s \in \GCal$ belongs to $\GCal_{A,4\tau+\delta}$ which implies $|\GCal_{A,4\tau+\delta}| \geq |\GCal| \geq \frac{7}{8} \cdot 2^M$. Moreover, we have $M = mk \geq cm\delta^{-2}$. Mapping $\{\pm1\}^{m\times k}$ to $\{\pm1\}^M$ yields the claim. Therefore, we conclude that 
\begin{equation*}
\mathfrak{m}_\textnormal{dc}(\XCal_\textnormal{op}(1),\FCal_\textnormal{op}(1,1),{\mathsf{O}},\delta)= \Omega(\min\{m,n\}\delta^{-2})
\end{equation*}
for any local oracle ${\mathsf{O}}$ on $\FCal_\textnormal{op}(1,1)$ and any $\delta\in(\frac{1}{\sqrt{\max\{m,n\}}},1)$. 

It suffices to prove that, for any local oracle $\mathsf{O}$ on $\FCal_\textnormal{op}(G,R)$ there exists a local oracle $\widetilde{\mathsf{O}}$ on $\FCal_\textnormal{op}(1,1)$ such that for any $\epsilon>0$, we have $\mathfrak{m}_\textnormal{dc}(\XCal_\textnormal{op}(R),\FCal_\textnormal{op}(G,R),\mathsf{O},\epsilon)\geq \mathfrak{m}_\textnormal{dc}(\XCal_\textnormal{op}(1),\FCal_\textnormal{op}(1,1),\widetilde{\mathsf{O}},\epsilon/GR)$. Indeed, we define $\mathscr{S}_{G,R}:\FCal_\textnormal{op}(1,1)\to\FCal_\textnormal{op}(G,R)$, $\mathscr{S}_{G,R}(F)(W):=GRF(W/R)$. This is a bijection and guarantees that $X \mapsto RX$ maps $\XCal_\textnormal{op}(1)$ bijectively onto $\XCal_\textnormal{op}(R)$ and $\widehat{X}$ is $(\epsilon/GR)$-optimal for $F$ if and only if $R\widehat{X}$ is $\epsilon$-optimal for $\mathscr{S}_{G,R}(F)$. Given a local first-order oracle $\mathsf{O}$ on $\FCal_\textnormal{op}(G,R)$, we define
\begin{equation*}
\widetilde{\mathsf{O}}_F(X) = \left(\tfrac{v}{GR},\tfrac{g}{G}\right), \textnormal{ where }(v,g)=\mathsf{O}_{\mathscr{S}_{G,R}(F)}(RX).
\end{equation*}
Since $g \in \partial\mathscr{S}_{G,R}(F)(RX)=G\cdot\partial F(X)$, we obtain that $\widetilde{\mathsf{O}}$ is the first-order oracle that locality follows from $\mathsf{O}$. Given any $\mathsf{A} \in \ACal(\mathsf{O})$, we construct $\widetilde{\mathsf{A}} \in \ACal(\widetilde{\mathsf{O}})$ that simulates $\mathsf{A}$ by mapping every query $W$ to $W/R$, rescaling the resulting oracle reply by $(GR,G)$, and mapping the terminal output $\widehat{W}$ to $\widehat{W}/R$. Thus, for every $F \in \FCal_\textnormal{op}(1,1)$, $\mathsf{T}_{\widetilde{\mathsf{A}}}(F,\epsilon/GR)=\mathsf{T}_{\mathsf{A}}(\mathscr{S}_{G,R}(F),\epsilon)$. Taking the infimum over algorithms and the supremum over distributions proves the desired result. 
\end{proof}

\section{Experiments}\label{sec:exp}
We conduct extensive experiments with nonsmooth hard instance, quadratic optimization, convolutional neural networks (CNNs) and LLM pretraining. The candidate approaches include SGD, inexact Muon (with either Newton-Schulz or PolarExpress) and MAGD (with either Newton-Schulz or PolarExpress). In each experiment, we tune the hyperparameters by grid search for all methods and also keep the step size and momentum schedules same for Muon and the Muon branch of MAGD. 

\subsection{Nonsmooth hard instance}
We consider a differentiable approximation of hard instances from Theorem~\ref{thm:fixed-momentum-failure} parameterized by $\delta \geq 0$. For $\delta=0$, they reduce to the original instances. For $\delta>0$, letting $p=W_{11}+W_{22}$ and
$q=W_{11}-W_{22}$, we define
\begin{equation*}
f_\delta(W) = A_1\left(\sqrt{p^2+\delta^2}-\delta\right)+(A_2+1)\delta\log\left(\tfrac{A_2+e^{q/\delta}}{A_2+1}\right)-q.
\end{equation*}
where the first term approximates $A_1|p|$, while the second term approximates $\max\{A_2q,0\}+\max\{-q,0\}$. We use a fixed momentum factor $\beta_t\equiv\beta=0.7$, $\mu=1.2671$ from the proof of Theorem~\ref{thm:fixed-momentum-failure}, $A_1 = 0.009283$, and $A_2=63.08$. We initialize $W_{11}$ and $W_{22}$ independently from standard normal distributions and set all of other entries as $0$.

We compare gradient descent (GD), fixed-momentum Muon, adaptive-momentum Muon, and MAGD for 50000 iterations over 100 initializations, using the exact polar updates. We search over step size schedules $x_t,y_t$ as follows, 
\begin{equation*}
x_t = \tfrac{x_0}{2}\left[1+\cos\left(\tfrac{\pi t}{T-1}\right)\right], \quad y_t= \tfrac{y_0}{(1+t)^a},
\end{equation*}
where $x_0 \in \{0.003,0.005,0.008,0.01,0.015,0.02,0.03,0.05,0.08\}$, $a \in \{0.5,0.6,0.75,0.9,1\}$ and $y_0 \in \{0.1,0.2,0.3,0.5,0.8\}$. For MAGD, we additionally tune $\bar{\eta}_t$ to follow the same schedules with $x_0,y_0\in \{0.02,0.03,0.035,0.04,0.05\}$, $\rho=0.5$, and $\gamma=1$, and use normalized gradient descent for the GD branch as in Theorem~\ref{thm:magd-convex-finite-time-rate}. For every method and value of $\delta$, we select the configuration that maximizes the fraction of initializations attaining an objective gap below $10^{-6}$.

Table~\ref{tab:counterexample} shows that fixed-momentum Muon cannot reach an objective gap below $10^{-6}$ for all runs when $\delta=0$. In contrast, adaptive-momentum Muon, GD and MAGD achieve improved performance. It is also worth noting that the performance of all methods improves as the parameter $\delta$ increases.
\begin{table}[!t]
\centering
\begin{tabular}{lcccc}
\toprule
$\delta$ & Muon (fixed $\beta_t=0.7$) & Muon ($\beta_t=\max\{1-256\eta_t,0\}$) & GD & MAGD (fixed $\beta_t^{\rm{M}}=0.7$)\\
\midrule
$0$ & $0\%$ & $85\%$ & $100\%$ & $100\%$ \\
$10^{-4}$ & $55\%$ & $100\%$ & $100\%$ & $100\%$ \\
$10^{-3}$ & $100\%$ & $100\%$ & $100\%$ & $100\%$ \\
$10^{-2}$ & $100\%$ & $100\%$ & $100\%$ & $100\%$ \\
$10^{-1}$ & $100\%$ & $100\%$ & $100\%$ & $100\%$ \\
\bottomrule
\end{tabular}
\caption{Percentage of the 100 initializations for which the objective gap is below $10^{-6}$.}
\label{tab:counterexample}
\end{table}

\subsection{Quadratic optimization}
We consider the quadratic optimization problems used in~\citet{Gonon-2026-Insights}. Taking $n=100$, the convex quadratic function is defined by 
\begin{equation*}
f(W) = \tfrac{1}{2} \langle W,AW\rangle,
\end{equation*}
where $A=Q\operatorname{diag}(s_1,\ldots,s_n)Q^\top$ is positive definite and $Q$ is random and orthogonal. After vectorizing $W$, the Hessian of $f$ is $I_n\otimes A$. Each eigenvalue $s_i$ of $A$ appears $n$ times in the spectrum of $f$. In Table~\ref{tab:quadratic-spectra}, we summarize seven spectrum families, where all of them use the same endpoints, $s_{\min}=10^{-3}$ and $s_{\max}=10$. Thus, every problem has condition number $10^4$ and only the shape of the eigenvalue distribution varies. We sort the eigenvalues in decreasing order and exactly enforce the endpoints.
\begin{table}[t]
\centering
\begin{tabular}{ll}
\toprule
Spectrum & Construction \\ \midrule
\texttt{max\_spiked} & One eigenvalue at $s_{\min}$, all others at $s_{\max}$. \\
\texttt{min\_spiked} & One eigenvalue at $s_{\max}$, all others at $s_{\min}$. \\
\texttt{uniform} & $s_i$ sampled uniformly on $[s_{\min},s_{\max}]$. \\
\texttt{gaussian} & $s_i$ sampled from a clipped Gaussian centered in $[s_{\min},s_{\max}]$. \\
\texttt{linear\_decay\_to\_max} & $s_i=s_{\max}-(s_{\max}-s_{\min})\sqrt{u_i}$, $u_i\sim\mathrm{Unif}[0,1]$. \\
\texttt{geometric\_decay\_to\_max} & $s_i=\max\{s_{\min},s_{\max}\cdot 0.9^{i-1}\}$. \\
\texttt{u\_shaped} & $s_i=s_{\min}+(s_{\max}-s_{\min})z_i$, $z_i\sim\mathrm{Beta}(0.2,0.2)$. \\ \bottomrule
\end{tabular}
\caption{Spectrum families used in the quadratic optimization experiment.}
\label{tab:quadratic-spectra}
\end{table}
\begin{figure}[!t]
\centering
\includegraphics[width=\linewidth]{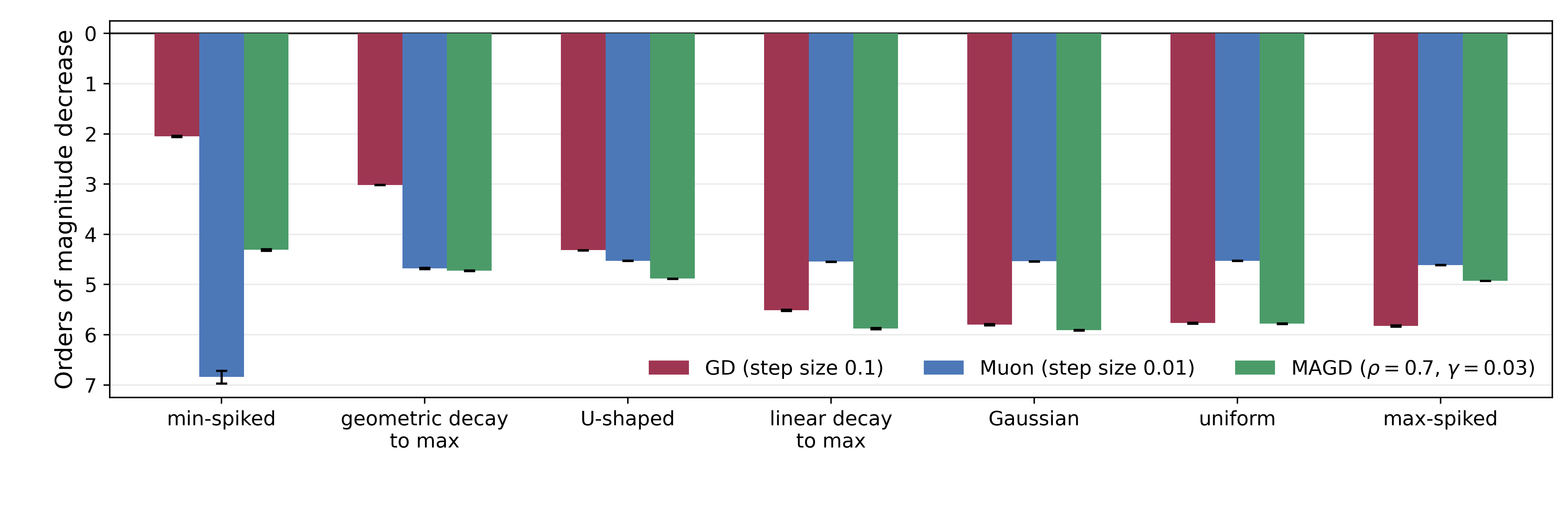} \vspace{-3em}
\caption{Orders of magnitude of loss decrease after 500 iterations on the quadratic functions.}
\label{fig:quadratic}
\end{figure}

We compare GD, Muon, and MAGD using PolarExpress~\citep{Amsel-2026-Polar} for orthogonalization. We sample $[W_0]_{ij} \sim \mathcal{N}(0,0.01)$ and run each method for $T=500$ iterations over 100 initializations on an NVIDIA 4060 Laptop GPU with 8 GB of VRAM. We then report the mean number of orders of magnitude by which the loss decreases (i.e., $\log_{10} \left(\tfrac{f(W_0)}{\min_{0 \leq t \leq T} f(W_t)}\right)$) over the random initializations, together with confidence intervals.

For hyperparameter tuning, we disable Muon momentum and perform a grid search over the constant step sizes in $\{0.1,0.01,0.001\}$ for both Muon and GD. The best step size is $0.1$ for GD and $0.01$ for Muon as observed in~\citet{Gonon-2026-Insights}. We run MAGD using these step sizes for its GD and Muon branches, respectively, with $\rho=0.7$ and $\gamma=0.03$.

We present the results in Figure~\ref{fig:quadratic}. GD is strong on several instances whose spectra are concentrated near the largest eigenvalues. Muon is robust on spectra for which GD is slow and performs extremely well on \texttt{min\_spiked}. MAGD combines these behaviors by adaptively mixing a GD branch and a Muon branch, thereby exhibiting best-of-both-worlds behavior on this benchmark. It outperforms the best of Muon and GD on five of the seven spectra. On each of the two spectra with a single outlier eigenvalue, it retains some of the gains achieved by the stronger baseline.

\subsection{Convolutional neural network}
We train a 2-million-parameter CIFARNET~\citep{Jordan-2024-CIFAR} on CIFAR-10~\citep{Krizhevsky-2009-Learning} for image classification using one NVIDIA GPU with 48 GB of VRAM. We split the training set into 40000 training samples and 10000 validation samples and evaluate the models on the test set of 10000 samples. Each run uses 200 optimization steps with a batch size of 2000, corresponding to 10 passes over the training set. We train with cross-entropy loss and label smoothing and report unsmoothed cross-entropy loss. We compare Muon, MAGD, and stochastic gradient descent with momentum (SGDM) for training the main convolutional layers, where none of them uses weight decay. All remaining trainable parameters other than the main convolutional layers are optimized by SGDM. When tuning the step size and momentum for these methods, we use 3 Newton-Schulz iterations for orthogonalization. The optimization problem is nonsmooth because the CNN uses max pooling.

We select the hyperparameters by maximizing validation accuracy averaged over 50 seeds. For Muon, we test the pair of step size and momentum factor in
\begin{align*}
&\{0.005, 0.01, 0.02, 0.05, 0.10, 0.18, 0.24, 0.30, 0.36, 0.40\} \times \\ &\{0.40, 0.50, 0.55, 0.60, 0.65, 0.70, 0.80, 0.85, 0.90, 0.95\}.
\end{align*}
For SGDM, we test step sizes in $\{0.0001, 0.0002, 0.0004, 0.0008, 0.001, 0.002, 0.004, 0.008, 0.01, 0.02\}$ with the same momentum grid. For MAGD, we retain Muon's selected step size and momentum factor and tune the MAGD-specific parameters under the same validation criterion. The selected step size and momentum factor are $0.001$ and $0.85$ for SGDM, and $0.24$ and $0.65$ for Muon. MAGD uses the selected Muon settings together with $\bar{\eta}_t=0.095$, $\gamma=0.19$, and $\rho=0.003$.

In Table~\ref{tab:cnn}, we report the validation and test losses and accuracies averaged over 100 random seeds. MAGD performs competitively with Muon, and its performance improves as the approximation becomes more accurate. Moreover, both MAGD and Muon outperform SGDM in this setting.
\begin{table}[!t]
\centering
\begin{tabular}{lcccc}
\hline
Method & Val. Loss & Val. Accuracy & Test Loss & Test Accuracy \\
\hline
MAGD (NS3) & $0.4312 \pm 0.0002$ & $0.9253 \pm 0.0002$ & $0.4271 \pm 0.0003$ & $0.9277 \pm 0.0002$ \\
Muon (NS3) & $0.4313 \pm 0.0003$ & $0.9253 \pm 0.0002$ & $0.4271 \pm 0.0003$ & $0.9279 \pm 0.0002$ \\
MAGD (NS5) & $0.4310 \pm 0.0003$ & $0.9260 \pm 0.0002$ & $0.4271 \pm 0.0002$ & $0.9282 \pm 0.0002$ \\
Muon (NS5) & $0.4310 \pm 0.0003$ & $0.9259 \pm 0.0002$ & $0.4273 \pm 0.0002$ & $0.9282 \pm 0.0001$ \\
MAGD (PolarExpress) & $0.4291 \pm 0.0002$ & $0.9266 \pm 0.0002$ & $0.4249 \pm 0.0002$ & $0.9288 \pm 0.0002$ \\
Muon (PolarExpress) & $0.4294 \pm 0.0003$ & $0.9266 \pm 0.0002$ & $0.4253 \pm 0.0003$ & $0.9286 \pm 0.0002$ \\
SGDM & $0.5337 \pm 0.0007$ & $0.8885 \pm 0.0004$ & $0.5322 \pm 0.0007$ & $0.8901 \pm 0.0003$ \\
\hline
\end{tabular}
\caption{Results are the mean $\pm$ standard error over 100 seeds for both validation and test. NS3 and NS5 mean three and five Newton-Schulz steps, respectively. PolarExpress uses five steps.} \label{tab:cnn}
\end{table}
\begin{figure}[!t]
\centering 
\includegraphics[width=1.0\linewidth]{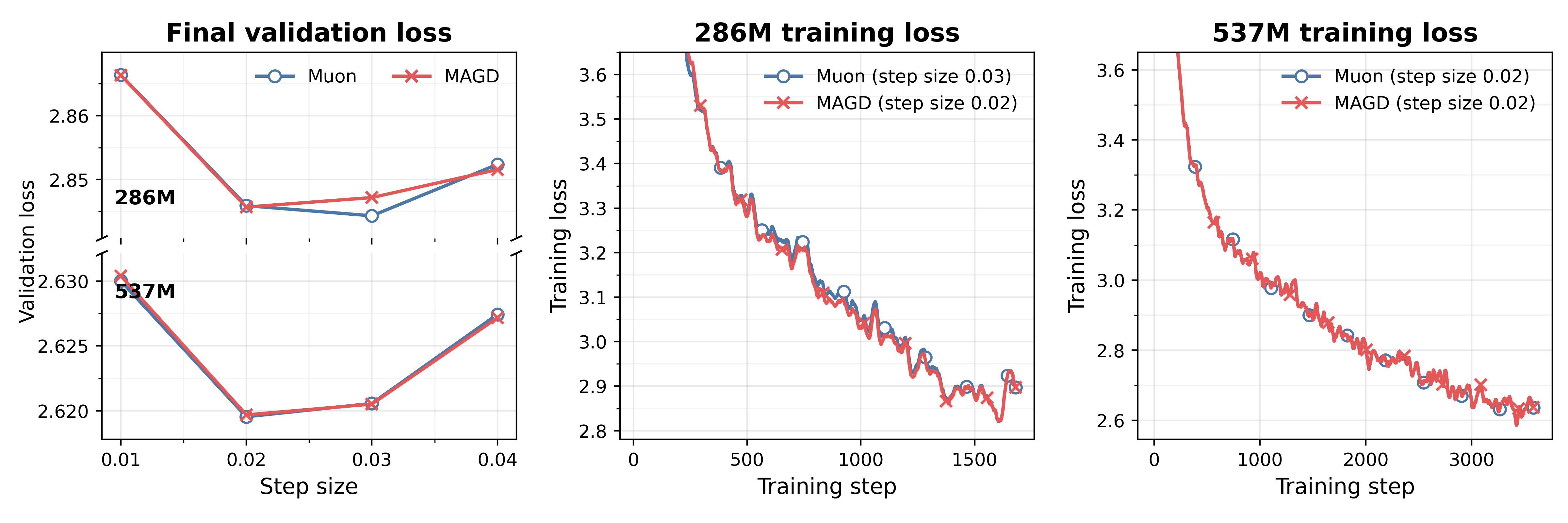}
\caption{Performance of Muon and MAGD on nanochat LLM pretraining.} \label{fig:LLM}
\end{figure}

\subsection{LLM pretraining}
We train 286M- and 537M-parameter nanochat transformers~\citep{Vaswani-2017-Attention, Karpathy-2025-Nanochat} on the NVIDIA ClimbMix dataset~\citep{Diao-2025-Nemotron} for 1680 and 3584 steps, respectively. Both settings use a vocabulary of 32768 tokens, 42M validation tokens, a global batch size of $524288=2^{19}$ tokens, and a sequence length of 2048. Each run for the 286M-parameter model uses one NVIDIA GPU with 48 GB of VRAM and 32 gradient-accumulation steps. Each run for the 537M-parameter model uses four 40 GB NVIDIA A100 GPUs  and 4 gradient-accumulation steps.

We use decoder-only transformers. Each transformer block contains self-attention with rotary position embeddings (RoPE)~\citep{Su-2024-Roformer} and $d \times 4d$ MLP layers with squared-ReLU activations. The 286M-parameter model has 12 layers, hidden dimension $d=768$, and 6 attention heads, whereas the 537M-parameter model has 16 layers, hidden dimension $d=1024$, and 8 attention heads. Both models use a head dimension of 128, learned value embeddings in alternating layers, and an untied head.

We compare Muon and MAGD for training the attention and MLP layers, using peak step sizes in $\{0.01,0.02,0.03,0.04\}$ with a warmup-stable-decay schedule. The weight decay used in Muon follows a cosine schedule, and the Muon branch of MAGD uses the same  schedule. We use five PolarExpress steps for orthogonalization. All remaining trainable parameters are optimized using AdamW~\citep{Loshchilov-2019-Decoupled} with fine-tuned step sizes and weight decay factors. For MAGD, we use $\bar{\eta}_t=0.01$, $\gamma=0.25$, and $\rho=10^{-4}$. Figure~\ref{fig:LLM} shows that MAGD is competitive with Muon and exhibits similar robustness across step sizes.

\section{Conclusion}\label{sec:conclu}
We studied orthogonalized momentum for locally Lipschitz matrix optimization and established a sharp distinction between fixed and coupled momentum schedules. For every fixed momentum factor $\beta \in [0,1)$, Muon can fail to approach the global optimal solution of a convex Lipschitz objective from almost every initialization, when the step sizes adapt to the full gradient history. In contrast, coupling the momentum factor to a diminishing step size through $\beta_t=1-\tau\eta_t$ can restore asymptotic convergence for inexact Muon under boundedness and vanishing polar error. In particular, every cluster point is stationary, the momentum converges to $0$, and the objective values converge. Under convexity, this value coincides with the global minimum. These results identify the interaction between fixed momentum and orthogonalization as a fundamental obstruction in nonsmooth optimization.

To obtain a finite-time guarantee, we proposed MAGD, which adaptively combines an inexact Muon branch with a gradient descent branch. It retains stationary cluster points and convergent objective values for nonconvex objectives and finds an $\epsilon$-optimal iterate within $O(\min\{m,n\}\epsilon^{-2})$ iterations for convex objectives. We derive the lower bound to demonstrate the optimal dependence on $\min\{m,n\}$ and $\epsilon$. The numerical results on the synthetic problems, image classification and LLM pretraining illustrate that MAGD safeguards the orthogonalized update while retaining its empirical performance. 

Several questions remain open. First, it is unclear whether Muon itself, without an auxiliary gradient branch, admits a uniform finite-time guarantee under an adaptive momentum and step-size schedule, even for convex Lipschitz objectives. Second, our asymptotic analysis requires the polar approximation error to vanish, whereas practical implementations use a fixed number of Newton-Schulz or PolarExpress iterations. Developing a theory that allows persistent approximation error would narrow the gap between theory and practice. Finally, extending the upper and lower bounds to other matrix-norm geometries would clarify which phenomena are specific to operator-norm optimization and which reflect more general principles of matrix-aware methods.

\section*{Acknowledgments}
The first author is partially supported by a start-up fund at the University of Hong Kong and by the Hong Kong Research Grants Council under ECS project 27301425. The second author is partially supported by a start-up fund and the Early Career Scholarship Support Grant at Columbia University. JZ thanks Derek Long for useful discussions. This work used GPU at NCSA and Indiana Jetstream2 through allocation CIS261175 from the Advanced Cyberinfrastructure Coordination Ecosystem: Services \& Support (ACCESS) program \citep{Boerner-2023-Access}, which is supported by U.S. National Science Foundation grants \#2138259, \#2138286, \#2138307, \#2137603, and \#2138296. We acknowledge the GPU support from Columbia University's Insomnia computer cluster. We sincerely appreciate Buzz High Performance Computing (\hyperlink{https://www.buzzhpc.ai}{\texttt{https://www.buzzhpc.ai}}, \texttt{info@buzzhpc.ai}) for providing computational resources and support for this work.

\bibliographystyle{plainnat}
\bibliography{ref}

\end{document}